\documentclass[psamsfonts,reqno]{amsart}
\usepackage{amssymb,eucal}

\usepackage{amscd}

\usepackage{graphicx}

\usepackage{mathrsfs}

\newcommand{\color}[6]{}

\theoremstyle{plain}

\newtheorem{lemma}{Lemma}[section]
\newtheorem{theorem}[lemma]{Theorem}
\newtheorem{proposition}[lemma]{Proposition}
\newtheorem{corollary}[lemma]{Corollary}

\newtheorem*{stat}{\name}
\newcommand{\name}{testing}

\theoremstyle{definition}
\newtheorem{definition}[lemma]{Definition}
\newtheorem{example}[lemma]{Example}

\theoremstyle{remark}
\newtheorem{remark}[lemma]{Remark}

\newcommand{\qedc}{{\qed}~{\rm Claim~{\theclaim}.}}
\newcommand{\qedsc}{{\qed}~{\rm Claim.}}

\numberwithin{equation}{section}

\newcommand{\pup}[1]{\textup{(}{#1}\textup{)}}

\newcommand{\set}[1]{\{#1\}}
\newcommand{\setm}[2]{\set{#1\mid#2}}
\newcommand{\Set}[1]{\left\{#1\right\}}
\newcommand{\Setm}[2]{\Set{#1\mid#2}}

\newcommand{\Pow}{\mathfrak{P}}
\newcommand{\Powf}[1]{[#1]^{<\omega}}
\newcommand{\dnw}{\mathbin{\downarrow}}
\newcommand{\upw}{\mathbin{\uparrow}}

\DeclareMathOperator{\supp}{supp}

\newcommand{\cC}{\mathcal{C}}

\newcommand{\sL}{\mathscr{L}}

\newcommand{\cS}{\mathcal{S}}
\newcommand{\cB}{\mathcal{B}}

\newcommand{\cV}{\mathcal{V}}
\newcommand{\cJ}{\mathcal{J}}
\newcommand{\cA}{\mathcal{A}}
\newcommand{\cD}{\mathcal{D}}

\newcommand{\cM}{\mathcal{M}}

\newcommand{\bx}{\boldsymbol{x}}

\newcommand{\bu}{\boldsymbol{u}}
\newcommand{\bv}{\boldsymbol{v}}

\DeclareMathOperator{\dom}{dom}
\DeclareMathOperator{\Con}{Con}
\DeclareMathOperator{\Conc}{Con_c}

\DeclareMathOperator{\Id}{Id}

\newcommand{\id}{\mathrm{id}}
\newcommand{\jz}{$(\vee,0)$}

\newcommand{\jzu}{$(\vee,0,1)$}
\newcommand{\jzs}{\jz-semi\-lat\-tice}

\newcommand{\jzus}{\jzu-semi\-lat\-tice}
\newcommand{\jzh}{\jz-ho\-mo\-mor\-phism}

\newcommand{\jzuh}{\jzu-ho\-mo\-mor\-phism}

\newcommand{\jzue}{\jzu-em\-bed\-ding}

\newcommand{\res}{\mathbin{\restriction}}

\newcommand{\module}[1]{|{#1}|}
\DeclareMathOperator{\card}{card}
\DeclareMathOperator{\Cond}{Cond}
\DeclareMathOperator{\Ob}{Ob}
\DeclareMathOperator{\Mor}{Mor}
\DeclareMathOperator{\M}{M}

\DeclareMathOperator{\crita}{crit}
\newcommand{\crit}[2]{\crita({{#1};{#2}})}
\newcommand{\critsym}[2]{\crita^{\mathrm{s}}({{#1};{#2}})}

\DeclareMathOperator{\Ide}{Id_e}
\DeclareMathOperator{\Ids}{Id_s}

\DeclareMathOperator{\cf}{cf}

\newcommand{\fU}{\mathfrak{U}}
\newcommand{\fF}{\mathfrak{F}}

\newcommand{\tosurj}{\mathbin{\twoheadrightarrow}}
\newcommand{\toinj}{\mathbin{\hookrightarrow}}

\usepackage[all]{xy}

\begin{document}

\title[Critical points]{Critical points of pairs of varieties of algebras}

\author[P.~Gillibert]{Pierre Gillibert}
\address{LMNO, CNRS UMR 6139\\
D\'epartement de Math\'ematiques, BP 5186\\
Universit\'e de Caen, Campus 2\\
14032 Caen cedex\\
France}
\email{pierre.gillibert@math.unicaen.fr}



\thanks{This paper is a part of the author's ``Doctorat de l'universit\'e de Caen'', prepared under the supervision of Friedrich Wehrung}

\date{\today}

\begin{abstract}
For a class~$\cV$ of algebras, denote by $\Conc\cV$ the class of all \jzs s isomorphic to the semilattice $\Conc A$ of all compact congruences of~$A$, for some~$A$ in~$\cV$. For classes~$\cV_1$ and~$\cV_2$ of algebras, we denote by $\crit{\cV_1}{\cV_2}$ the smallest cardinality of a \jzs\ in $\Conc\cV_1$ which is not in $\Conc\cV_2$ if it exists, $\infty$ otherwise. We prove a general theorem, with categorical flavor, that implies that for all finitely generated congruence-distributive varieties~$\cV_1$ and~$\cV_2$, $\crit{\cV_1}{\cV_2}$ is either finite, or $\aleph_n$ for some natural number $n$, or $\infty$. We also find two finitely generated modular lattice varieties~$\cV_1$ and~$\cV_2$ such that $\crit{\cV_1}{\cV_2}=\aleph_1$, thus answering a question by J.~T\r{u}ma and F.~Wehrung.
\end{abstract}

\maketitle

\section{Introduction}\label{S:Intro}
We denote by~$\Con A$ (resp., $\Conc A$) the lattice (resp., \jzs) of all congruences (resp., compact congruences) of an algebra~$A$. For a homomorphism~$f\colon A\to B$ of algebras, we denote by $\Con f$ the map from~$\Con A$ to~$\Con B$ defined by the rule
 \[
 (\Con f)(\alpha)=\text{congruence of }B\text{ generated by }\setm{(f(x),f(y))}{(x,y)\in \alpha},
 \]
for every $\alpha\in\Con A$. We also denote by~$\Conc f$ the restriction of~$\Con f$ from~$\Conc A$ to~$\Conc B$. This defines a functor $\Conc$ from the category of algebras of a fixed similarity type to the category of all \jzs s, moreover $\Conc$ preserves direct limits.

A \emph{lifting} of a \jzs\ $S$ is an algebra~$A$ such that $\Conc A\cong S$. Given a variety~$\cV$ of algebras, \emph{the compact congruence} class of~$\cV$, denoted by $\Conc\cV$, is the class of all \jzs s isomorphic to $\Conc A$ for some $A\in\cV$. As illustrated by~\cite{PTW}, even the compact congruence classes of small varieties are complicated objects.

Let~$\cV$ be a variety of algebras, let~$\cD$ be a diagram  of \jzs s and $(\vee,0)$-homomorphisms. A \emph{lifting} of~$\cD$ in~$\cV$ is a diagram $\cA$ of~$\cV$ such that the composite $\Conc\circ\cA$ is naturally isomorphic to~$\cD$.

To a poset~$I$ and a diagram $\vec S=(S_i,\varphi_i^j)_{i\leq j\text{ in }I}$ of \jzs s, we shall associate a \jzs~$C$, which is a subdirect product of the~$S_i$s.

We shall establish a set of results that can be loosely summed up as follows:
 \begin{quote}\em
 In the `good cases', if $C$ has a lifting in~$\cV$, then~$\vec S$ has a lifting in~$\cV$; and conversely. 
 \end{quote}
The \jzs~$C$ is not defined from~$\vec S$ alone, but from what we shall call a \emph{norm-covering} of~$I$ (Definition~\ref{D:normcovering}). By definition, a norm-covering of~$I$ is a pair $(U,\module{\cdot})$, where~$U$ is a so-called \emph{supported poset} (Definition~\ref{D:KerSupp}) and~$\module{\cdot}\colon U\to I$ is an isotone map. We shall write $C=\Cond(\vec S,U)$, and call~$C$ a \emph{condensate} of~$\vec S$ (cf. Section~\ref{S:Condensate}). The assignment $\vec S\mapsto\Cond(\vec S,U)$ can be naturally extended to a functor.

Among the above-mentioned `good' cases is the case where~$I$ is a \emph{well-founded  tree} (i.e., all principal lower subsets are well-founded chains). Hence we can associate liftings of \jzs s with liftings of diagrams of \jzs s indexed by trees (Corollary~\ref{C:liftinginitialstep}). By iterating this result finitely many times, we obtain similar results for diagrams indexed by finite products of trees (Corollary~\ref{C:lifting-with-nth-succesor-cardinal}). In particular (cf. Corollary~\ref{C:AllsemilatticesToAllDiagramProductOfTrees}), that if all \jzs s of a `good' class of \jzs s $\cS$ have a lifting in a variety~$\cV$, then every diagram of $\cS$, indexed by finite products of well-founded trees, has a lifting in~$\cV$. In particular, using the result, proved by W.\,A. Lampe in~\cite{Lamp82}, that \emph{every \jzus\ is isomorphic to~$\Conc G$ for some groupoid~$G$}, we prove in Corollary~\ref{C:generalizedLamp} that \emph{every diagram of \jzus s and \jzuh s, indexed by a finite poset, has a lifting in the variety of groupoids}. This extends to all finite poset-indexed diagrams the result, proved in \cite{Lamp06} for one zero-separating arrow, of simultaneous representation.

Funayama and Nakayama proved in \cite{FuNa42} that $\Conc L$ is distributive for any lattice~$L$. However, our result above cannot be extended to \jzus s replaced by \emph{distributive} \jzus s and groupoids replaced by \emph{lattices}. This is due to the negative solution to the Congruence Lattice Problem, obtained by F.~Wehrung in~\cite{CLP}, that gives a distributive \jzus\ that is not isomorphic to~$\Conc L$ for any lattice~$L$.

A somehow strange, but unavoidable, feature of our proof is that the condensate construction builds objects of larger cardinality. For example, in order to be able to lift diagrams indexed by (at most) countable chains of (at most) countable \jzs s, we need to be able to lift \jzs s of cardinality~$\aleph_1$.

Another interesting problem is the comparison of congruence classes of varieties of algebras.
Given two varieties~$\cV_1$ and~$\cV_2$ of algebras, the \emph{critical point} of~$\cV_1$ and~$\cV_2$, denoted by $\crit{\cV_1}{\cV_2}$, is the smallest cardinality of a \jzs\ in $\Conc(\cV_1)-\Conc(\cV_2)$ if it exists, or $\infty$, otherwise (i.e., if $\Conc\cV_1\subseteq\Conc\cV_2$). Denote by~$M_n$ the lattice of length two with~$n$ atoms and by $\cM^{0,1}_n$ the variety of bounded lattices generated by~$M_n$, for any positive integer~$n$.
M. Plo\v s\v cica gives in~\cite{Ploscica03} a characterization of \jzus s of cardinality $\aleph_1$ in $\Conc\cM_n^{0,1}$, and he proves that the result is independent of~$n$.
Moreover, M. Plo\v s\v cica also proves in~\cite{Ploscica00} that if we denote by $L$ the free lattice of $\cM_{n+1}^{0,1}$ with $\aleph_2$ generators, then $\Conc L$ has no lifting in $\cM_n^{0,1}$. (M. Plo\v s\v cica proves his results for varieties of bounded lattices, but for those negative results the difference between bounded and unbounded is inessential.) This implies that $\crit{\cM_m^{0,1}}{\cM_n^{0,1}}=\aleph_2$ for all integers $m>n\ge 3$.

One corollary of our main result is that if the critical point between two varieties~$\cV_1$ and~$\cV_2$ of algebras with countable similarity types is greater than $\aleph_n$, then all diagrams of countable \jzs s indexed by products of~$n$ finite chains which are liftable in~$\cV_1$ are also liftable in~$\cV_2$.

In Corollary~\ref{C:critpointifFGCDVisalephn} we prove that the critical point between a locally finite variety and a finitely generated congruence-distributive variety is either finite, or~$\aleph_n$  for some natural number $n$, or~$\infty$. Moreover in Section~\ref{S:critpointaleph1} we give two finitely generated varieties of modular lattices with critical point $\aleph_1$, which solves negatively Problem~5 in \cite{CLPSurv}. However, we still do not know whether there exists a pair of varieties of lattices with critical point $\aleph_n$ with $n\geq 3$.

\section{Basic concepts}\label{S:Basic}
We denote by $\dom f$ the domain of any function $f$. We write $\Pow(X)$ the set of all subsets of $X$ and $\Powf{X}$ the set of all finite subsets of $X$, for every set $X$. We denote by $\kappa^+$ the cardinal successor of $\kappa$ and $\kappa^{+n}$ the $n^{\mathrm{th}}$ successor of $\kappa$, and we denote $\cf\kappa$ the cofinality of $\kappa$, for every cardinal $\kappa$.

A \emph{poset} is a partially ordered set. We denote by $P^-$ (resp., $P^=$) the set of all non-minimal (resp., non-maximal) elements in a poset $P$. For $i, j\in P$ let $i\prec j$ hold, if $i<j$ and there is no $k\in P$ with $i<k<j$, in this case $i$ is called a \emph{lower cover} of $j$. If $j$ has exactly one lower cover, we denote it by $j_*$.
We put
\[Q\dnw X=\setm{p\in Q}{(\exists x\in X)(p\le x)},\qquad Q\upw X=\setm {p\in Q}{(\exists x\in X)(p\ge x)},\]
for any $X, Q\subseteq P$, and we will write $\dnw X$ (resp., $\upw X$) instead of $P\dnw X$ (resp., $P\upw X$) in case $P$ is understood. We shall also write $\dnw p$ instead of $\dnw\set{p}$, and so on, for $p\in P$. A poset $I$ is \emph{lower finite}, if $I\dnw i$ is finite for all $i\in I$. A subset $X$ of $P$ is a \emph{lower subset} if $P\dnw X=X$. An \emph{ideal} of $P$ is a nonempty, upward directed, lower subset of $P$. We denote by $\Id P$ the set of all ideals of $P$, partially ordered  
by inclusion. We will often identify $a$ with $P\dnw a$, where $a\in P$, and identify $P$ with $\setm{P\dnw a}{a\in P}\subseteq\Id P$. A \emph{tree} is a poset $T$ with a smallest element such that $T\dnw t$ is a chain for each $t\in T$. We denote by $\M(L)$ the set of all completely meet-irreducible elements in a lattice~$L$.

For an algebra~$A$ and $P\subseteq A^2$, denote by $\Theta_A(P)$ the smallest congruence of~$A$ that contains~$P$. We put $\Theta_A(x,y)=\Theta_A(\set{(x,y)})$, for all $x,y\in A$. Let $X\subseteq A$ we denote by $\Conc^X(A)=\setm{\Theta_A(P)}{P\in\Powf{X^2}}$ the set of all congruences of $A$ finitely generated by parameters in $X$.

Let $(A_i)_{i\in I}$ be a family of algebras of the same similarity type, let $(\theta_i)_{i\in I}\in(\Con A_i)^I$; the \emph{congruence product of $(\theta_i)_{i\in I}$} is the congruence defined by:
\[
\prod_{i\in I}\theta_i = \Setm{(x,y)\in \left(\prod_{i\in I}A_i\right)^2}{\forall i\in I,\ (x_i,y_i)\in\theta_i }.
\]

We denote by $x/\theta$ the equivalence class of $x$ modulo~$\theta$, where $\theta$ is an equivalence relation on a set~$A$ and $x\in A$. We shall often write $X/\theta=\setm{x/\theta}{x\in X}$, for any subset~$X$ of~$A$. The \emph{canonical embedding} from $X/(\theta\cap(X\times X))$ into~$A/\theta$ sends $x/(\theta\cap(X\times X))$ to~$x/\theta$, for each $x\in X$. We shall often identify~$X/\theta$ and $X/(\theta\cap(X\times X))$.

For a category $\cC$, we write $\Ob\cC$ the class of all objects of $\cC$ and $\Mor\cC$ the class of all morphisms in~$\cC$.

For categories $I$ and $J$, denote by $J^I$ the category whose objects are the functors from $I$ to $J$ and whose arrows are the natural transformations. Let $I, J$, and $\cS$ be categories, let $\cD\colon J\to\cS^I$ be a functor. We can define a functor:
\begin{align*}
\widehat{\cD}\colon I\times J & \to\cS\\
(i,j) &\mapsto\cD(j)(i) & & \text{for all $(i,j)\in\Ob(I\times J)$}\\
(f,g) &\mapsto\cD(g)_{i'}\circ\cD(j)(f) & &\text{for all $(f\colon i\rightarrow i',\ g\colon j\rightarrow j')\in\Mor(I\times J)$,}
\end{align*}
where $\cD(g)=(\cD(g)_k)_{k\in\Ob I}$.
Conversely, given a functor $\cD\colon I\times J\to\cS$, we can define a functor $\widetilde{\cD}\colon J\to\cS^I$ by:
\begin{align*}
\widetilde{\cD}(j)\colon I &\to\cS\\
i &\mapsto \cD(i,j), & &\text{for all $i\in\Ob I$}\\
f &\mapsto \cD(f,\id_j), & &\text{for all $f\in\Mor I$}
\end{align*}
which is a functor, for all $j\in\Ob J$, and
\[
\widetilde{\cD}(g) = (\cD(\id_i,g))_{i\in\Ob I} \colon \widetilde{\cD}(j) \to \widetilde{\cD}(k)
\]
which is a natural transformation, for all $(g\colon j\to k)\in\Mor J$.

We shall identify every poset $P$ with the category whose objects are the elements of $P$, and that has exactly one arrow, then denoted by $(i\le j)$, from $i$ to $j$, just in case $i\le j$ in $P$.

Let $\cS$ be a class of \jzs s, let~$\cV$ be a class of algebras of the same similarity type, let~$J$ be a category. A \emph{lifting in~$\cV$} of a functor $\cD\colon J\to\cS$ is a functor $\cA\colon J\to \cV$ such there exists a natural isomorphism $\Conc\circ\cA\to\cD$. In this case we say that $\cA$ is a lifting of~$\cD$ in~$\cV$.

Let~$J$ be a category. We put $i\unlhd j$, if there exists an arrow $f\colon i\to j$ of $J$, for all $i$ and $j$ in $\Ob J$. This relation is reflexive and transitive.

Let $I$ and $\cS$ be categories, let $\cD\colon I\to \cS$ be a functor. We denote by $\varinjlim\cD$ a colimit of $\cD$ if it exists. Strictly speaking, it is a cocone of $\cS$, however, we often identify it with its underlying object in~$\cS$.  Similarly, if all colimits indexed by $I$ exist, we consider $\varinjlim\colon\cS^I\to\cS$ as a functor. Colimits indexed by upwards directed posets are often called \emph{direct limits}.

It is well-known that any variety of algebras, viewed as a category, has all small colimits (small here means that the index category is small).

A variety of algebras is \emph{congruence-distributive} if each of its members has a distributive congruence lattice.

\section{A L\"owenheim-Skolem type property}\label{S:LSP}

\begin{definition}\label{D:LS}
Let $U$ be a poset, let~$J$ be a small category, and $\vec\kappa=(\kappa_u)_{u\in U}$ be a family of cardinals. A class~$\cV$ of algebras of the same similarity type is \emph{$(U,J,\vec\kappa)$-L\"owenheim-Skolem}, if for any functor $\cA\colon J\to\cV$ and for any family $(\alpha_u^j)_{u\in U}^{j\in\Ob J}$ of congruences, with $\alpha_u^j\in\Con\cA(j)$, such that $\sum_{j\in\Ob J}\card\Conc(\cA(j)/\alpha_u^j)<\kappa_u$ for all $u\in U$, there exists
a family $(B_u^j)_{u\in U}^{j\in\Ob J}$ of algebras such that:
\begin{enumerate}
\item The algebra $B_u^j$ is a subalgebra of $\cA(j)$ for all $u\in U$ and all $j\in \Ob J$.
\item The algebra $B_u^j/\alpha_u^j$ belongs to~$\cV$ for all $u\in U$ and all $j\in\Ob J$.
\item The containment $B_u^j\subseteq B_v^j$ holds for all $u\le v$ in $U$ and all $j\in\Ob J$.
\item The containment $\cA(f)(B_u^j)\subseteq B_u^k$ holds for every $u\in U$ and every morphism $f\colon j\to k$ in~$J$.
\item The morphism $\Con(q_u^j)$ is an isomorphism, where $q_u^j$ denotes the canonical embedding $B_u^j/\alpha_u^j\toinj \cA(j)/\alpha_u^j$.
\item The inequality $\sum_{j\in \Ob J}\card{B_u^j}<\kappa_{u}$ holds for all $u\in U$.
\end{enumerate}
\end{definition}

The following result appears in \cite[Theorem 10.4]{UnivAlgebra}.

\begin{lemma}\label{L:Condcompcongruences}
$\Theta_B(x,y)\leq\bigvee_{i<m}\Theta_B(x_i,y_i)$ if{f} there are a positive integer $n$, a list $\vec z$ of parameters from~$B$, and terms $t_1$, \dots, $t_n$ such that
\begin{align*}
x&=t_1(\vec x,\vec y,\vec z),\\
y&=t_n(\vec x,\vec y,\vec z),\\
t_j(\vec y,\vec x,\vec z)&=t_{j+1}(\vec x,\vec y,\vec z)\quad(\text{for all }j<n).
\end{align*}
\end{lemma}




\begin{definition}
Let $\kappa$ be a cardinal. An algebra is \emph{locally $<\kappa$} if every finitely generated subalgebra is of cardinality $<\kappa$. The definition of \emph{locally $\le\kappa$} is similar. An algebra is \emph{locally finite} if it is locally $<\aleph_0$.

A variety of algebras is \emph{locally $<\kappa$} (resp., \emph{locally $\le\kappa$}) if all its members are locally $<\kappa$ (resp., \emph{locally $\le\kappa$}).
\end{definition}

\begin{remark}
Let $\sL$ be a similarity type. Every $\sL$-algebra is locally $\le\card\sL$. Let $\kappa$ be a cardinal, let $\sL\subseteq\sL'$ be similarity types, let $(E,\sL')$ be an algebra such that $(E,\sL)$ is locally $\le\kappa$, then $(E,\sL')$ is locally $\le\kappa+\card(\sL'-\sL)$.

Let $\kappa$ be a cardinal. If $E$ is a locally $\le\kappa$ algebra, then every subalgebra of $E$, generated by at most $\kappa$ elements, has at most $\kappa$ elements.
\end{remark}

The following lemma is proved using an argument similar to the one in the usual proof of the L\"owenheim-Skolem Theorem.

\begin{lemma}\label{L:quasiLS}
Let $\sL$ be a similarity type. Let $E$ be a $\sL$-algebra, let $Q\subseteq E$. Let $(\sL_i)_{i\in I}$ be a family of sub-similarity types of $\sL$. Let $\kappa$ be an infinite cardinal. If $(E,\sL)$ is locally $\le\kappa$, then there exists a subalgebra $(F,\sL)$ of $(E,\sL)$ such that:
\begin{enumerate}
\item The containment $Q\subseteq F$ is satisfied,
\item The inequality $\card F\le\kappa+\card Q+\card I$ holds,
\item The morphism $\Conc q_i\colon\Conc(F,\sL_i)\to\Conc(E,\sL_i)$ is one-to-one, where $q_i\colon (F,\sL_i)\to (E,\sL_i)$ denotes the inclusion map, for all $i\in I$.
\end{enumerate}
\end{lemma}

\begin{proof}
Let $A_0$ be the subalgebra of $(E,\sL)$ generated by $Q$. As $E$ is locally $\le\kappa$, we have $\card A_0\le\kappa+\card Q$. Let $n<\omega$. Assume that we have constructed subalgebras $A_0\subseteq \dots\subseteq A_n$ of $(E,\sL)$ of cardinality at most $\kappa+\card Q+\card I$, such that for all $0\le u<v\le n$, for all $i\in I$, for all $m\in\mathbb{N}$, for all $x$, $y$, $x_1,\dots,x_m$, $y_1,\dots,y_m$ in $A_u$ we have the following equivalence
\begin{equation*}
\Theta_{(E,\sL_i)}(x,y)\le\bigvee_{1\le k\le m}\Theta_{(E,\sL_i)}(x_k,y_k)\ \Longleftrightarrow\ \Theta_{(A_v,\sL_i)}(x,y)\le\bigvee_{1\le k\le m}\Theta_{(A_v,\sL_i)}(x_k,y_k).
\end{equation*}

Let $i\in I$. Let $x$, $y$, $x_1,\dots,x_m$, $y_1,\dots,y_m$ in~$X_n$, such that the inequality
\begin{equation}\label{E:LSeq1}
\Theta_{(E,\sL_i)}(x,y)\le\bigvee_{1\le k\le m}\Theta_{(E,\sL_i)}(x_k,y_k)\text{\quad is satisfied}.
\end{equation}
Lemma~\ref{L:Condcompcongruences} implies that there are a positive integer $r$, a list $\vec z$ of parameters from~$E$, and terms $t_1$, \dots, $t_r$ such that
\begin{align}\label{E:LSeq2}
x& = t_1(\vec x,\vec y,\vec z),\notag\\ 
y& = t_r(\vec x,\vec y,\vec z),\notag\\
t_j(\vec y,\vec x,\vec z)&= t_{j+1}(\vec x,\vec y,\vec z)\quad(\text{for all }j<r).
\end{align}
So we can construct $X\subseteq E$ such that $A_n\subseteq X$, $\card X\le \card A_n+\card I+\kappa$ and for all $i\in I$ and all $x$, $y$, $x_1,\dots,x_m$, $y_1,\dots,y_m$ in~$X_n$ that satisfy \eqref{E:LSeq1}, there are a positive integer $r$, a list $\vec z$ of parameters from~$X$ and terms $t_1$, \dots, $t_r$ that satisfy \eqref{E:LSeq2}. Let $A_{n+1}$ be the subalgebra of $(E,\sL)$ generated by $X$. As $(E,\sL)$ is locally $\le\kappa$, we have $\card A_{n+1}\le\card X+\kappa\le \card A_n+\card I+\kappa\le\kappa+\card Q+\card I$. Moreover, by construction, our induction hypothesis is satisfied.

So there exists a sequence $(A_n)_{n<\omega}$ of subalgebras $(E,\sL)$ of cardinality at most $\kappa+\card Q+\card I$ such that for all $0<u<v$, for all $i\in I$, for all $m\in\mathbb{N}$, for all $x$, $y$, $x_1,\dots,x_m$, $y_1,\dots,y_m$ in $A_u$ the following equivalence holds:
\begin{equation*}
\Theta_{(E,\sL_i)}(x,y)\le\bigvee_{1\le k\le m}\Theta_{(E,\sL_i)}(x_k,y_k)\ \Longleftrightarrow\ \Theta_{(A_v,\sL_i)}(x,y)\le\bigvee_{1\le k\le m}\Theta_{(A_v,\sL_i)}(x_k,y_k).
\end{equation*}
Put $F=\bigcup_{n<\omega}A_n$, we have $Q\subseteq A_0\subseteq F$ and $\card F\le\sum_{n<\omega}\card A_n\le\kappa+\card Q+\card I$. It is easy to check that for all $i\in I$ and for all $m\in\mathbb{N}$, for all $x$, $y$, $x_1,\dots,x_m$, $y_1,\dots,y_m$ in $F$ the following equivalence holds:
\begin{equation*}
\Theta_{(E,\sL_i)}(x,y)\le\bigvee_{1\le k\le m}\Theta_{(E,\sL_i)}(x_k,y_k)\ \Longleftrightarrow\ \Theta_{(F,\sL_i)}(x,y)\le\bigvee_{1\le k\le m}\Theta_{(F,\sL_i)}(x_k,y_k).
\end{equation*}
Thus the morphism $\Conc q_i\colon\Conc(F,\sL_i)\to\Conc(E,\sL_i)$ is one-to-one.
\end{proof}

The following lemma is a generalization of the L\"owenheim-Skolem theorem to diagrams of algebras.

\begin{lemma}\label{L:LS}
Let $\kappa$ be a cardinal. Let $\sL$ be a similarity type, let~$\cV$ be a variety of $\sL$-algebras locally $\le\kappa$, let~$J$ be a small category, let $\cA\colon J\to\cV$ be a functor, let~$\alpha_j$ be a congruence of $\cA(j)$, and let~$Q_j$ be a subset of $\cA(j)$ for all $j\in\Ob J$. Then there exists a family $(B_j)_{j\in \Ob J}$ of algebras such that:
\begin{enumerate}
\item The algebra $B_j$ is a subalgebra of $\cA(j)$ for all $j\in\Ob J$.
\item The containment $\cA(f)(B_j)\subseteq B_k$ holds for every arrow $f\colon j\to k$ of $J$.
\item The morphism $\Con(q_j)$ is an isomorphism, where $q_j$ denotes the canonical embedding $B_j/\alpha_j\toinj \cA(j)/\alpha_j$, for all $j\in\Ob J$.
\item The following inequality holds:
\[
\card B_j\le  \kappa + \card \Mor (J\res j)+\sum_{i\unlhd j}\Big(\card\Conc(\cA(i)/\alpha_i) + \card Q_i\Big),\quad\text{for all $j\in\Ob J$,}
\]
where $J\res j$ denotes the full subcategory of $J$ with $\setm{i\in\Ob J}{i\unlhd j}$ as class of objects.
\item The containment $Q_j\subseteq B_j$ holds for all $j\in\Ob J$.
\end{enumerate}
\end{lemma}

\begin{proof} Let $(Q'_j)_{j\in\Ob J}$ be a family of sets such that:
\begin{enumerate}
\item The set $Q'_j$ is a subset of $\cA(j)$.
\item The equality $\Conc(\cA(j)/\alpha_j) = \Conc^{Q'_j/\alpha_j}(\cA(j)/\alpha_j)$ holds.
\item The inequality $\card Q'_j\le \aleph_0 + \card \Conc(\cA(j)/\alpha_j) + \card Q_j$ holds.
\item The containment $Q_j\subseteq Q'_j$ holds.
\end{enumerate}
for all $j\in\Ob J$

Fix a family $(x_j)_{j\in\Ob J}\in \prod_{j\in J}\cA(j)$. Let $I$ be a finite subset of $\Ob J$, we denote by $\overline{I}$ the full subcategory of $J$ with class of objects $I$. Put $T_I=\bigsqcup_{j\in I}\cA(j)$, where $\bigsqcup$ denotes the disjoint union. Put $\sL_I=\Mor\overline{I}\sqcup \bigcup_{j\in I}(\set{j}\times\sL)$. We shall extend~$\sL_I$ to a similarity type (i.e., assign an arity to each element of~$\sL_I$) and endow~$T_I$ with a structure of a $\sL_I$-algebra. For each $n$-ary operation symbol $\ell\in\sL$ and each $j\in\Ob J$, we say that $(j,\ell)$ is a $n$-ary operation symbol, and we put:
\[
(j,\ell)^{T_I}(a_1,a_2,\dots,a_n)=\begin{cases}
\ell^{\cA(j)}(a_1,a_2,\dots,a_n) &\text{if $a_1,a_2,\dots,a_n\in \cA(j)$,}\\
x_j &\text{Otherwise,}
\end{cases}
\]
for all $a_1,a_2,\dots,a_n\in T_I$. Every $f\in\Mor \overline{I}$ will be a unary operation symbol, and for $f\colon i\to j$ we put:
\[
f^{T_I}(a)=\begin{cases}
\cA(f)(a) &\text{for all $a\in \cA(i)$,}\\
x_j &\text{for all $a\in T_I-\cA(i)$.}
\end{cases}
\]
Put $\sL_I'=\bigcup_{j\in I}(\set{j}\times\sL)\subseteq\sL_I$. We first show that $(T_I,\sL_I')$ is locally $\le\kappa$. Let~$X$ be a finite subset of~$T_I$. Put~$X_j=\set{x_j}\cup(X\cap\cA(j))$ for all $j\in I$. Let $Y_j$ be the subalgebra of $(\cA(j),\sL)$ generated by $X_j$, for all $j\in I$. As $\cA(j)$ is locally $\le\kappa$ and~$X_j$ is finite, we get $\card Y_j\le\kappa$, for all $j\in I$. Put $Y=\bigsqcup_{j\in I}Y_j$, then $Y$ is a subalgebra of $(T_I,\sL_I')$ and $Y\supseteq X$. It follows that $(T_I,\sL_I')$ is locally $\le\kappa+\card I$. Moreover we have $\sL_I-\sL_I'=\Mor\overline{I}$, so $(T_I,\sL_I)$ is locally $\le\kappa+\card\Mor\overline{I}$.

Put $\sL_j=\set{j}\times\sL$, for all $j\in\Ob J$. The similarity type $\sL_j$ is a sub-similarity type of $\sL_I$, for all $I\in\Powf{\Ob J}-\set{\emptyset}$ and all $j\in I$. Applying Lemma~\ref{L:quasiLS}, arguing by induction on $\card I$, we construct a family $(T'_I,\sL_I)_{I\in\Powf{\Ob J}-\set{\emptyset}}$ of algebras such that:
\begin{enumerate}
\item the algebra $(T_I',\sL_I)$ is a subalgebra of $(T_I,\sL_I)$,
\item the morphism $\Conc q_j^{I}\colon \Conc (T_I',\sL_j)\to\Conc (T_I,\sL_j)$ is one-to-one, where $q_j^{I}\colon (T_I',\sL_j)\to (T_I,\sL_j)$ denotes the inclusion map, for all $j\in I$.
\item the containment $\bigsqcup_{i\in I}Q_i' \subseteq T'_I$ holds,
\item the containment $T'_K\subseteq T'_I$ holds,
\item the inequality $\card T'_I\le \kappa + \sum_{i\in I}\card Q_i'+\card\Mor\overline{I}$ holds,
\end{enumerate}
for all finite nonempty subsets $K\subseteq I$ of $\Ob J$.

Let $I$ be a finite nonempty subset of $\Ob J$, let $j\in I$. Put $B_j^I=\cA(j)\cap T'_I$. We consider:
\begin{align*}
q_j^I\colon (T_I',\sL_j)&\to (T_I,\sL_j),\\
p_j^I\colon (B_j^I,\sL_j)&\to (\cA(j),\sL_j),\\
s_j^I\colon (B_j^I,\sL_j)&\to (T_I',\sL_j),\\
t_j^I\colon (\cA(j),\sL_j)&\to (T_I,\sL_j),
\end{align*}
the inclusion maps. The map $\Conc q_j^I$ is one-to-one. The following diagram is commutative:
\[ 
\begin{CD}
(T_I',\sL_j)  @>q_j^I>>  (T_I,\sL_j)\\
@A{s_j^I}AA                    @AA{t_j^I}A\\
(B_j^I,\sL_j) @>>p_j^I>   (\cA(j),\sL_j)
\end{CD}
\]
Let $\theta$ be a congruence of $(B_j^I,\sL_j)$, it is easy to check that $\theta\cup\id_{T_I'}$ is a congruence of $(T_I',\sL_j)$. Thus $\Conc s_j^I$ is one-to-one. Hence $\Conc p_j^I$ is one-to-one. The following statements hold:

\begin{enumerate}
\item The morphism $\Conc p_j^I\colon (B_j^I,\sL)\to (\cA(j),\sL)$ is one-to-one, where we denote by $p_j^I\colon (B_j^I,\sL)\to (\cA(j),\sL)$ the inclusion map.
\item The containment $Q_j' \subseteq B_j^I$ holds,
\item The containment $B_j^K\subseteq B_j^I$ holds,
\item $\card B_j^I\le \kappa + \sum_{i\in I}\card Q_i' + \card\Mor\overline{I}$,
\item $\cA(f)(B_i^I)\subseteq B_j^I$ for all $f\colon i\to j$ in $\overline{I}$,
\end{enumerate}
for each finite nonempty subset $K\subseteq I$ of $\Ob J$ and each $j\in K$.

The subset $B_j= \bigcup_{I\in\Powf{\Ob (J\dnw j)}-\set{\emptyset}}B_j^I$ is a directed union of the algebras~$B_j^I$, for $I\in\Powf{\Ob (J\dnw j)}-\set{\emptyset}$. Moreover, the following statements hold for each $j\in\Ob J$:
\begin{itemize}
\item The map $\Conc p_j\colon\Conc(B_j/\alpha_j)\to\Conc(\cA(j)/\alpha_j)$ is one-to-one, where $p_j\colon B_j/\alpha_j\toinj\cA(j)/\alpha_j$ denotes the canonical embedding.
\item The containment $Q_j' \subseteq B_j$ holds. So $\Conc(\cA(j)/\alpha_j) = \Conc^{B_j/\alpha_j}(\cA(j)/\alpha_j)$, and so $\Conc q_j\colon\Conc (B_j/\alpha_j)\to\Conc(\cA(j)/\alpha_j)$ is an isomorphism.
\item The following inequalities hold:
\begin{align*}
\card B_j&\le \sum_{I\in\Powf{\Ob (J\res j)}-\set{\emptyset}}\left(  \kappa + \sum_{i\in I}\card Q_i' + \card\Mor\overline{I}  \right)\\
&\le \kappa + \sum_{i\unlhd j}\card Q_i'+ \sum_{I\in\Powf{\Ob (J\res j)}-\set{\emptyset}}\Big( \card \Mor\overline{I}\Big)\\
&\le \kappa + \sum_{i\unlhd j} \Big( \card Q_i+\Conc(\cA(i)/\alpha_i) \Big) +\card\Mor(J\res j)
\end{align*}
\item $\cA(f)(B_i)\subseteq B_j$ for all $f\colon i\to j$ in $J$.\qed
\end{itemize}
\renewcommand{\qed}{}
\end{proof}

\begin{lemma}\label{L:LS2}
Le $\lambda$ be an infinite cardinal. Let $\sL$ be a similarity type, let~$\cV$ be a locally~$\le\lambda$ variety of $\sL$-algebras, let $U$ be a poset, let~$J$ be a small category, and let $\vec\kappa=(\kappa_u)_{u\in U}$ be a family of cardinals such that
\begin{enumerate}
\item the inequality $\lambda+\card\Mor J<\kappa_u$ holds for all $u\in U$, 
\item for any family $(\kappa^j_u)_{u\in U}^{j\in \Ob J}$ of cardinals such that $\kappa_u^j<\kappa_u$ for all $u\in U$ and all $j\in\Ob J$, the inequality $\sum_{v\le u}\sum_{j\in\Ob J}\kappa^j_v<\kappa_u$ holds.
\end{enumerate}
Then~$\cV$ is $(U,J,\vec\kappa)$-L\"owenheim-Skolem.
\end{lemma}

\begin{proof}
Let $\cA\colon J\to\cV$ be a functor, let $(\alpha_u^j)_{u\in U}^{j\in\Ob J}$ be a family of congruences with all $\alpha_u^j\in\Con\cA(j)$, such that $\sum_{j\in\Ob J}\card\Conc(\cA(j)/\alpha_u^j)<\kappa_u$ for all $u\in U$. We can define a functor $\cA'\colon J\times U \to \cV$ by
\begin{align*}
(j,u)&\mapsto\cA(j) & &\text{for all $(j,u)\in\Ob (J\times U)$,}\\
(f\colon i\to j,\ u\le v)&\mapsto \cA(f) & &\text{for all $(f\colon i\to j,\ u\le v)\in\Mor (J\times U)$.}
\end{align*}
Moreover, $\alpha_u^j$ is a congruence of $\cA'(j,u)$ for all $(j,u)\in\Ob (J\times U)$. So, by Lemma~\ref{L:LS}, there exists a family $(B_u^j)_{(j,u)\in\Ob(J\times U)}$ of algebras such that:
\begin{enumerate}
\item The algebra $B_u^j$ is a subalgebra of $\cA'(j,u)$ for all $(j,u)\in\Ob(J\times U)$.
\item The containment $\cA'(f,u\le v)(B_u^j)\subseteq B_v^k$ holds for every arrow $(f\colon j\to k,\ u\le v)$ of $J\times U$.
\item The morphism $\Con(q_u^j)$ is an isomorphism, where $q_u^j$ denotes the canonical embedding $B_u^j/\alpha_u^j\toinj\cA'(j,u)/\alpha_u^j$, for all $(j,u)\in\Ob (J\times U)$.
\item The following inequality holds, for all $(j,u)\in\Ob(J\times U)$:
\begin{align*}
\card B_u^j\le & \kappa + \card \Mor \big((J\times U)\res (j,u)\big)\\
&+\sum_{(i,v)\unlhd (j,u)\text{ in $J\times U$}}\Big(\card\Conc(\cA'(i,v)/\alpha_v^i)\Big).
\end{align*}
\end{enumerate}
The statements (1)--(5) of Definition \ref{D:LS} are satisfied. Moreover:
\[\card \Mor \big((J\times U)\res (j,u)\big) \le \kappa + \card \Mor J + \card (U\dnw u)<\kappa_u,\]
for all $(j,u)\in\Ob (J\times U)$. As $\card\Conc(\cA'(i,u)/\alpha_u^i)=\card\Conc(\cA(i)/\alpha_u^i)<\kappa_u$, the following inequalities hold:
\begin{align*}
\sum_{(i,v)\unlhd (j,u)\text{ in $J\times U$}} \Big(\card\Conc(\cA'(i,v)/\alpha_v^i)\Big)& \le\sum_{v\le u} \sum_{i\in\Ob J} \Big(\card\Conc(\cA(i)/\alpha_v^i)\Big)\\
&<\kappa_u.
\end{align*}
Thus $\card B_u^j<\kappa_u$, for every $u\in U$ and for every $j\in\Ob J$. So, using again the assumptions of the lemma, the following inequality holds:
\[\sum_{j\in\Ob J}\card B_u^j<\kappa_u,\quad\text{for all $u\in U$}.\tag*{\qed}\]
\renewcommand{\qed}{}
\end{proof}

\begin{lemma}\label{L:conc-on-CD-FG-variety-is-one-to-one}
Let~$\cV$ be a finitely generated congruence-distributive variety of algebras. Let $S$ be a finite \jzs. Then there exist, up to isomorphism, at most finitely many $A\in\cV$ such that $\Conc A\cong S$. Moreover, all such~$A$ are finite.
\end{lemma}

\begin{proof}
As~$\cV$ is a finitely generated congruence-distributive variety of algebras, there exist, by J\'onsson's Lemma, only finitely many, up to isomorphism, subdirectly irreducible algebras in~$\cV$, and they are all finite.
Let $A\in\cV$ such that $\Conc A\cong S$. Recall that $\M(\Con A)$ denote the set of all completely meet-irreducible elements of $\Con A$, hence $A/\theta$ is subdirectly irreducible for all $\theta\in\M(\Con A)$.
As~$A$ embeds into the product $A\toinj\prod_{\theta\in \M(\Con A)}A/\theta$, and $\M(\Con A)\cong\M(\Id S)$, the conclusion follows.
\end{proof}

\begin{lemma}\label{L:LS-FGCDV}
Let~$\cV$ be a finitely generated congruence-distributive variety of algebras, let $U$ be a lower finite poset, let~$J$ be a finite poset, put $\kappa_u=\aleph_0$ for all $u\in U$. Then~$\cV$ is $(U,J,\vec\kappa)$-L\"owenheim-Skolem.
\end{lemma}

\begin{proof}
Let $\cA\colon J\to\cV$ be a functor, let $(\alpha_u^j)_{u\in U}^{j\in\Ob J}$ be a family of congruences, with all $\alpha_u^j\in\Con(\cA(j))$, such that $\card\Conc(\cA(j)/\alpha_u^j)<\aleph_0$ for all $u\in U$ and all $j\in J$.

By Lemma~$\ref{L:conc-on-CD-FG-variety-is-one-to-one}$, $\cA(j)/\alpha_u^j$ is finite for all $u\in U$ and all $j\in J$. Let $Q_u^j$ be a finite subset of $\cA(j)$ such that $\cA(j)/\alpha_u^j=\setm{q/\alpha_u^j}{q\in Q_u^j}$ for all $j\in J$ and all $u\in U$. Let $B_u^j$ be the subalgebra of $\cA(j)$ generated by $\bigcup\setm{\cA(i,j)(Q_v^i)}{v\le u\text{ and } i\le j}$ for all $j\in J$ and all $u\in U$. As~$\cV$ is finitely generated, all objects of~$\cV$ are locally finite, and so $B_u^j$ is finite for all $j\in J$ and all $u\in U$. Moreover the following statements hold:
\begin{enumerate}
\item The algebra $B_u^j$ is a subalgebra of $\cA(j)$ for all $u\in U$ and all $j\in \Ob J$.
\item The algebra $B_u^j/\alpha_u^j=\cA(j)/\alpha_u^j$ belongs to~$\cV$ for all $u\in U$ and all $j\in\Ob J$.
\item The containment $B_u^j\subseteq B_v^j$ holds for all $u\le v$ in $U$ and all $j\in\Ob J$.
\item The containment $\cA(j,k)(B_u^j)\subseteq B_u^k$ holds for every $u\in U$ and every $j\le k$ in~$J$.
\item The canonical embedding $q_u^j\colon B_u^j/\alpha_u^j\toinj A^j/\alpha_u^j$ is an isomorphism, so $\Con(q_u^j)$ is an isomorphism.
\item The inequality $\sum_{j\in \Ob J}\card{B_u^j}<\aleph_0$ holds for all $u\in U$.\qed
\end{enumerate}
\renewcommand{\qed}{}
\end{proof}

\section{Kernels, supported posets, and norm-coverings}\label{S:norm-covering}
\begin{definition}\label{D:KerSupp}
A finite subset $V$ of a poset $U$ is a \emph{kernel}, if for every $u\in U$,
there exists a largest element $v\in V$ such that $v\le u$. We denote
this element by $V\cdot u$.

We say that $U$ is \emph{supported}, if every finite subset of $U$ is
contained in a kernel of~$U$.
\end{definition}

We denote by $V\cdot \bu$ the largest element of $V\cap\bu$, for every
kernel $V$ of $U$ and every ideal~$\bu$ of~$U$.
As an immediate application of the finiteness of kernels, we obtain
the following.

\begin{lemma}\label{L:kersupp}
Any intersection of a nonempty collection of kernels of a poset $U$ is a
kernel of $U$.
\end{lemma}

\begin{example}\label{E:supp}
Let $\kappa$ be a cardinal, we put $T_\kappa=\kappa\sqcup\set{\bot}$ with order defined by $x\le y$ if either $x=y$ or $x=\bot$. Then $T_\kappa$ is a supported poset, and the kernels of $T_\kappa$ are all the finite subsets containing $\bot$.
\end{example}

\begin{definition}\label{D:normcovering}
A \emph{norm-covering} of a poset $I$ is a pair $(U,\module{\cdot})$, where $U$ is a supported poset and $\module{\cdot}\colon U\to I$, $u\mapsto\module{u}$ is an order-preserving map.

A \emph{sharp ideal} of $(U,\module{\cdot})$ is an ideal $\bu$ of $U$ such that $\setm{\module{v}}{v\in \bu}$ has a largest element, we denote this element by $\module{\bu}$. For example, for every $u\in U$, the principal ideal $U\dnw u$ is sharp; we shall often identify~$u$ and~$U\dnw u$. We denote by $\Ids(U,\module{\cdot})$ the set
of all sharp ideals of $(U,\module{\cdot})$, partially ordered by inclusion.

A sharp ideal $\bu$ of $(U,\module{\cdot})$ is \emph{extreme},
if there is no sharp ideal $\bv$ with $\bv>\bu$ and $\module{\bv}=\module{\bu}$. We
denote by $\Ide(U,\module{\cdot})$ the set of all extreme ideals of $(U,\module{\cdot})$.

The norm-covering is \emph{tight} if the map
$\Ide(U,\module{\cdot})\dnw \bu\to I\dnw\module{\bu}$, $\bv\mapsto \module{\bv}$ is a poset isomorphism for all $\bu\in \Ide(U,\module{\cdot})$.

Let $\vec\kappa=(\kappa_i)_{i\in I}$ be a family of cardinal numbers.
We say that $(U,\module{\cdot})$ is \emph{$\vec\kappa$-compatible},
if for every order-preserving map $F\colon\Ide(U,\module{\cdot})\to\mathfrak{P}(U)$ such that $\card F(\bu)<\kappa_{\module{\bu}}$ for all $\bu\in \Ide(U,\module{\cdot})^=$, there exists an order-preserving map $\sigma\colon I\to \Ide(U,\module{\cdot})$ such that:
\begin{enumerate}
\item The equality $\module{\sigma(i)}=i$ holds for all $i\in I$.
\item The containment $F(\sigma(i))\cap\sigma(j)\subseteq\sigma(i)$ holds for all $i\leq j$ in~$I$.
\end{enumerate}
We will say `$\kappa$-compatible' instead of $\vec \kappa$-compatible in case $\kappa_i=\kappa$ for all $i\in I$.
\end{definition}

Observe that the condition $(2)$ implies that $V\cdot\sigma(i)=V\cdot\sigma(j)$, for any $i\le j$ in~$I$ and any kernel $V$ of $U$ contained in $F(\sigma(i))$.

\begin{example}
Let $T_\kappa$ as defined in Example~\ref{E:supp}. We consider $\set{0,1}$ the two-element chain. We put :
\begin{align*}
\module{\cdot}\colon T_\kappa &\to\set{0,1}\\
x&\mapsto\module{x}=\begin{cases}
0  & \text{if $x=\bot$,}\\
1  & \text{otherwise.}
\end{cases}
\end{align*}
Thus $(T_\kappa,\module{\cdot})$ is a norm-covering of $\set{0,1}$. Moreover :
\[
\Ide(T_\kappa,\module{\cdot})=\setm{\dnw u}{u\in T_\kappa}=\set{\set{\bot}}\cup\setm{\set{\bot,\alpha}}{\alpha\in\kappa}\cong T_\kappa.
\]
Let $f\colon\Ide(T_\kappa,\module{\cdot})\to\Pow(\kappa)$ such that $\card f(\bu)<\kappa$ for all $\bu\in\Ide(T_\kappa,\module{\cdot})^=$. Hence $\card f(\set{\bot})<\kappa$. Let $\alpha\in \kappa-f(\set{\bot})$. Let $\sigma(0)=\set{\bot}$ and $\sigma(1)=\set{\bot,\alpha}$, we have $\module{\sigma(0)}=0$ and $\module{\sigma(1)}=1$. Moreover $f(\sigma(0))\cap\sigma(1)\subseteq\sigma(0)$. Hence $(T_\kappa,\module{\cdot})$ is a $\kappa$-compatible norm-covering of $\set{0,1}$.
\end{example}

The following construction is a generalization of this example, but we give a norm-covering of a tree instead of one of the two-element chain.

\begin{proposition}\label{P:Existkappcomp}
Let $T$ be a well-founded tree and let $\vec\kappa=(\kappa_t)_{t\in T}$
and $(\kappa'_t)_{t\in T^-}$ be families of infinite cardinals such
that for any $t\in T^-$ the following statements hold:
\begin{enumerate}
\item If $t$ has a lower cover, then $\kappa'_t\ge\kappa_{t_*}$.
\item If $t$ has no lower cover, then for any family $(\kappa''_s)_{s<t}$ of cardinals such that $\kappa''_s<\kappa_s$ for any $s<t$,
the inequality $\sum_{s<t}\kappa''_s<\kappa'_t$ holds.
\end{enumerate}
Then there exists a tight $\vec\kappa$-compatible norm-covering $(U,\module{\cdot})$ of $T$ such that $\card U=\sum_{t\in T^-}\kappa'_t$.
\end{proposition}

\begin{proof}
We denote by $\bot$ the least element of $T$, and we put $\phi(t)=(T\dnw t)-\set{\bot}$, for any $t\in T$. We put:
 \[
 U=\bigcup\Setm{\prod_{t\in C}\kappa'_t}{\text{$C$ is a finite chain of $T^-$}};
 \]
We view the elements of U as (partial) functions and ``to be greater" means ``to extend".

We put $\module{u}=\bigvee\dom u$, for any $u\in U$. We should note that the chain $C$ may be empty (in the definition of $U$), and $\module{\emptyset}=\bot$.

We prove that $U$ is supported. Let $V$ be a finite subset of $U$. Put:
\[Y_s=\setm{u_s}{u\in V\text{ and }s\in\dom u},\quad\text{for all $s\in T^-$}\]
and put $D=\setm{s\in T^-}{Y_s\not=\emptyset}$, hence $D=\bigcup_{u\in V}\dom u$. Put:
\[W=\setm{u\in U}{\dom u\subseteq D\text{ and }(\forall t\in\dom u)(u_t\in Y_t)}\]

The sets $D$, and $Y_s$, for all $s\in T^-$,
are finite, so $W$ is finite. As $u_s\in Y_s$ for all $u\in V$ and $s\in\dom u$, $V$ is contained in $W$.

Let $u\in U$ and $S=\setm{s\in \dom u}{u_s\in Y_s}$, then $u\res S\in  
W$. The containment $\dom v\subseteq S$ holds for all $v\in W\dnw u$, so $u\res S$ is
the largest element of $W$ smaller than~$u$, and so $W$ is a kernel
of $U$ containing $V$. Thus $(U,\module{\cdot})$ is a norm-covering of~$T$.

The set $\setm{x\res P}{\text{$P$ finite subset of $\phi(t)$}}$ is an extreme ideal of $(U,\module{\cdot})$, for all $t\in T$ and all $x\in\prod_{s\in\phi(t)}\kappa'_s$. We identify this ideal with $x$.
Moreover, all the extreme ideals of $(U,\module{\cdot})$ are of this form.
Thus $(U,\module{\cdot})$ is a tight norm-covering of~$T$.

Let $F\colon\Ide(U,\module{\cdot})\to\mathfrak{P}(U)$ be an order-preserving map such that $\card F(\bu)<\kappa_{\module{\bu}}$ for all $\bu\in \Ide(U,\module{\cdot})^=$. Put:
\[
F_t(\bu)=\{v_t\mid v\in F(\bu)\text{ and } t\in\dom v\},\quad\text{for all $t\in T^-$ and all $\bu\in \Ide(U,\module{\cdot})$.}
\]
Thus $\card F_t(\bu)\le\card F(\bu)<\kappa_{\module{\bu}}$, for all $\bu\in \Ide(U,\module{\cdot})^=$ and all $t\in T^-$.

Let $S$ be a lower subset of $T^-$ and let $x\in\prod_{t\in S}
\kappa'_t$ such that $x_t\notin F_t(x\res\phi(s))$ for all $s<t$ in
$S$. Let $t\in T^-$ such that $t\notin S$ and $\phi(t)-\set{t}\subseteq S$.
If $t$ has a lower cover, then $\card \bigcup_{s<t} F_{t}(x\res \phi(s))=\card F_{t}(x\res \phi(t_*))<\kappa_{t_*}\le\kappa'_t$. If $t$ has no lower cover, then $\card \bigcup_{s<t} F_{t}(x\res \phi(s))\le\sum_{s<t}\card F_{t}(x\res \phi(s))<\kappa'_t$.
In both cases, we can extend $x$ to $S\cup\set{t}$ by picking $x_t\notin F_t(x\res\phi(s))$ for all $s<t$ in~$T$.

As $T$ is well-founded, we can construct by induction $x\in\prod_{t\in T^-}\kappa'_t$ such that $x_t\notin F_t(x\res\phi(s))$ for all $s<t$ in~$T$. The map $\sigma\colon T\to \Ide(U,\module{\cdot})$, $t\mapsto x\res\phi(t)$ is order-preserving, and $\module{\sigma(t)}=t$, for all $t\in T$.

Let $s<t$ in $T$. Let $u\in F(x\res \phi(s))\cap (x\res \phi(t))$, and let $C=\dom u$. So $C\subseteq\phi(t)$, and $u=x\res C$. Let $s<r\le t$, by construction $x_r\notin F_r(x\res\phi(s))$, so $r\notin C$. Thus $C\subseteq\phi(s)$, and so $u=x\res C$ belongs to $x\res\phi(s)$.
\end{proof}

\begin{corollary}\label{C:Existnormcovering}
Let $T$ be a well-founded tree and let $\kappa$ be an infinite
cardinal such that $\card{T}\le\kappa$ and $\card(\dnw t)<\cf\kappa$ for all $t\in T$. Then
there exists a tight $\kappa$-compatible norm-covering $(U,\module{\cdot})$ of $T$ such that $\card U=\kappa$.
\end{corollary}

\begin{proof}
Put $\kappa_t=\kappa'_t=\kappa$, for any $t\in T$. The assumptions of
Proposition~\ref{P:Existkappcomp} are clearly satisfied.
\end{proof}

\section{Condensates}\label{S:Condensate}

\begin{definition}
Let $I$ be a poset, let $(U,\module{\cdot})$ be a norm-covering of $I$, and let $\vec{A}=(A_i,f_{i,j})_{i\le j\text{ in }I}$ be a diagram of algebras of the same similarity type.
\begin{itemize}
\item A \emph{support} $V$ of $a\in\prod_{u\in U}A_{\module{u}}$ is a
kernel $V$ of $U$ such that $a_u=f_{\module{V\cdot u},\module{u}}(a_{V\cdot u})$ for all $u\in U$.
\item We put:
\[\Cond_U(\vec{A},V)=\Setm{a\in\prod_{u\in U}A_{\module{u}}}{\text{$V$ is a support of $a$}},\quad\text{for any kernel $V$ of $U$.}\]
The \emph{condensate} of $\vec A$ with respect to $U$ is:
\[\Cond(\vec{A},U)=\bigcup\Setm{\Cond_U(\vec{A},V)}{\text{$V$ is a kernel of $U$}}.\]
\item We denote by $\supp a$ the smallest support of~$a$, and we call it \emph{the support} of~$a$.
\end{itemize}
By Lemma~\ref{L:kersupp} the support of $a$ exists, for all $a\in\Cond(\vec{A},U)$.
\end{definition}

\begin{lemma}\label{L:prop_de_base_cond}
With the notations of the previous definition, the following  
statements hold.
\begin{enumerate}
\item The set $\Cond_U(\vec{A},V)$ is a subalgebra of $\prod_{u\in U}
A_{\module{u}}$, for each kernel $V$ of $U$.
\item The containment $\Cond_U(\vec{A},V)\subseteq \Cond_U(\vec{A},W)$ holds, for all kernels $V$ and $W$ of $U$ such that $V\subseteq W$.
\item The set $\Cond(\vec{A},U)$ is a subalgebra of $\prod_{u\in U}A_{\module{u}}$, and it is the directed union
of the algebras $\Cond_U(\vec{A},V)$, with $V$ kernel of $U$.
\item The morphism $\pi_V\colon\Cond_U(\vec{A},V)\to \prod_{v\in V}A_{\module{v}}$, $a\mapsto a\res V$ is an isomorphism, for any kernel $V$ of $U$.
\item The algebra $\Cond(\vec{A},U)$ is a directed union of finite products of the $A_i$s.
\item The morphism $\pi_u\colon\Cond(\vec{A},U)\to A_{\module{u}}$, $a\mapsto a_u$ is onto, for all $u\in U$.
\item The map:
\begin{align*}
\pi_{\bu}\colon\Cond(\vec{A},U)&\to A_{\module{\bu}}\\
a&\mapsto f_{\module{\supp(a)\cdot \bu},\module{\bu}}(a_{\supp(a)\cdot \bu})
\end{align*}
is a surjective morphism of algebras, for all $\bu\in\Ids(U,\module{\cdot})$.
Furthermore $\pi_{\bu}(a)=f_{\module{V\cdot \bu},\module{\bu}}(a_{V\cdot \bu})$,
for any kernel~$V$ of~$U$ and any $a\in \Cond_U(\vec{A},V)$.
\end{enumerate}
\end{lemma}

\begin{proof}
The statements $(1)$, $(2)$, and $(3)$
are immediate. The morphism $\pi_V$ in $(4)$ is clearly one-to-one.
Let $x\in\prod_{v\in V}A_{\module{v}}$, put $a_u=f_{\module{V\cdot u},\module{u}}(x_{V\cdot u})$, for all $u\in U$.
Then $V$ is a support of $a$, and $a\res V=x$. So $\pi_V$ is an isomorphism.
The statement $(5)$ follows from $(4)$ and $(3)$. The statement~$(6)$ follows from $(4)$.

Now we verify (7). Let $\bu\in\Ids(U,\module{\cdot})$. Let $V$ be a
kernel of $U$ and let $a\in \Cond_U(\vec{A},V)$, then:
\begin{align*}
\pi_{\bu}(a)&=f_{\module{\supp(a)\cdot \bu},\module{\bu}}(a_{\supp(a)\cdot \bu})\\
&=f_{\module{V\cdot \bu},\module{\bu}}(f_{\module{\supp(a)\cdot \bu},\module {V\cdot \bu}}(a_{\supp(a)\cdot \bu}))\\
&=f_{\module{V\cdot \bu},\module{ \bu}}(a_{V\cdot \bu})
\end{align*}
This, together with $(3)$, shows that $\pi_{\bu}$ is a morphism of
algebras. Let $v\in \bu$ such that $\module{v}=\module{\bu}$, and let $V$
be a kernel of $U$ such that $v\in V$. Then $\module{V\cdot \bu}=\module{\bu}$, and $\pi_{\bu}\res\Cond_U(\vec{A},V)=f_{\module{V\cdot \bu},\module 
{ \bu}}\circ\pi_{V\cdot\bu}\res\Cond_U(\vec{A},V)=\pi_{V\cdot\bu}\res\Cond_U(\vec{A},V)$ is
surjective.
\end{proof}
We shall call the map $\pi_{\bu}$ above the \emph{canonical projection} from $\Cond(\vec A,\bu)$ to $A_{\module{\bu}}$.

\begin{proposition}\label{P:CondIsFunctor}
Let~$\cV$ be a class of algebras closed under finite products and under directed unions,
let $I$ be a poset, let $(U,\module{\cdot})$ be a norm-covering of $I$, let
$\vec{A}=(A_i,f_{i,j})_{i\le j\text{ in }I}$ and
$\vec{B}=(B_i,g_{i,j})_{i\le j\text{ in }I}$ be two objects of $\cV^I$,
and let $\vec{h}=(h_i)_{i\in I}\colon\vec{A}\to\vec{B}$ be an arrow
of $\cV^I$. Then there are morphisms of algebras:
\begin{align*}
\Cond_U(\vec{h},V)\colon\Cond_U(\vec{A},V) &\to\Cond_U(\vec{B},V)\\
(a_u)_{u\in U} & \mapsto (h_{\module{u}}(a_u))_{u\in U}, & &\text{for any kernel $V$ of $U$}
\end{align*}
and
\begin{align*}
\Cond(\vec{h},U)\colon\Cond(\vec{A},U) &\to\Cond(\vec{B},U)\\
(a_u)_{u\in U} & \mapsto (h_{\module{u}}(a_u))_{u\in U}
\end{align*}
Moreover, $\Cond(-,U)\colon\cV^I\to\cV$  is a functor.
\end{proposition}

\section{Liftings}\label{S:Liftings}
In this section, let $\cS$ be a class of \jzs s, closed under finite products and directed unions, let $I$ be a poset, let $\vec\kappa=(\kappa_i)_{i\in I}$ be a family of cardinal numbers, let $(U,\module{\cdot})$ be a $\vec\kappa$-compatible norm-covering of $I$, and let~$\cV$ be a class of algebras of the same similarity type.

\begin{proposition}\label{P:thetaisanideal}
Let $\vec{D}=(D_i,\phi_{i,j})_{i\le j\text{ in }I}$ be an object of $\cS^I$, let $\bu\in\Ids(U,\module{\cdot})$, let $\pi_{\bu}^{\vec{D}}\colon\Cond(\vec{D},U)\tosurj D_{\module{\bu}}$ be the canonical projection. Then the subset
\[\theta_{\bu}^{\vec{D}}=\setm{a\in \Cond(\vec{D},U)}{ \pi_{\bu}^{\vec{D}}(a)=0},\]
is an ideal of $\Cond(\vec{D},U)$, and
$\Id(\pi_{\bu}^{\vec{D}})\res \upw\theta_{\bu}^{\vec{D}}\colon\upw\theta_{\bu}^{\vec{D}}\to\Id(D_{\module{\bu}})$ is an isomorphism, where we abbreviate $(\Id\Cond(\vec{D},U))\upw\theta_{\bu}^{\vec{D}}$ by $\upw\theta_{\bu}^{\vec{D}}$.
\end{proposition}

\begin{proof}
The morphism $\rho_{\bu}=\Id(\pi_{\bu}^{\vec{D}})$ is surjective and $\rho_{\bu}(\theta_{\bu}^{\vec{D}})=0$, so $\rho_{\bu}\res\upw\theta_{\bu}^{\vec{D}}$ is surjective.

Fix $v\in \bu$ such that $\module{v}=\module{\bu}$. Let $L,L'\in\upw\theta_{\bu}^{\vec{D}}$ such that $\rho_{\bu}(L)\subseteq\rho_{\bu}(L')$, we must prove that $L\subseteq L'$. Let $a\in L$. As $\pi_{\bu}^{\vec{D}}(a)\in\rho_{\bu}(L')$, there exists $a'\in L'$ such that $\pi_{\bu}^{\vec{D}}(a)\le\pi_{\bu}^{\vec{D}}(a')$. Let $V$ be a common support of $a$ and $a'$ such that $v\in V$. So $\module{\bu}=\module{v}\le\module{V\cdot \bu}\le\module{\bu}$, and so $\module{V\cdot \bu}=\module{\bu}$, and hence $\phi_{\module{V\cdot \bu},\module{\bu}}=\id$. Therefore,
\[
a_{V\cdot \bu} = \phi_{\module{V\cdot \bu},\module{\bu}}(a_{V\cdot \bu}) = \pi_{\bu}^{\vec{D}}(a)\le \pi_{\bu}^{\vec{D}}(a') = \phi_{\module{V\cdot \bu},\module{\bu}}(a'_{V\cdot \bu})=a'_{V\cdot \bu}.
\]
Put:
\[
b_w=\begin{cases}
a_w &\text{If $V\cdot w\not=V\cdot \bu$}\\
0 &\text{If $V\cdot w=V\cdot \bu$}\\
\end{cases}
,\quad\text{for all $w\in U$.}
\]
The set $V$ is a support of $b$, and $\pi_{\bu}^{\vec{D}}(b)=\phi_{\module{V\cdot \bu},\module{\bu}}(b_{V\cdot\bu})=0$.

Let $w\in U$. If $V\cdot w\not=V\cdot \bu$, then $a_w=b_w\le a'_w\vee b_w$. If $V\cdot w=V\cdot \bu$, then $\module{\bu}=\module{V\cdot \bu}=\module{V\cdot w}\le\module{w}$. Thus:
\[
a'_w=\phi_{\module{V\cdot w},\module{w}}(a'_{V\cdot w})=\phi_{\module{\bu},\module{w}}(\phi_{\module{V\cdot \bu},\module{\bu}}(a'_{V\cdot \bu}))=\phi_{\module{\bu},\module{w}}(\pi_{\bu}^{\vec{D}}(a')),
\]
and, similarly, $a_w=\phi_{\module{\bu},\module{w}}(\pi_{\bu}^{\vec{D}}(a))$. As $\pi_{\bu}^{\vec{D}}(a)\le\pi_{\bu}^{\vec{D}}(a')$, we obtain that $a_w\le a'_w$. So we have proved that $a\le b\vee a'$. As $b\in\theta_{\bu}^{\vec{D}}\in L'$ and $a'\in L'$, it follows that $a\in L'$. Hence $L\subseteq L'$, and $\rho_{\bu}$ is an embedding.
\end{proof}

\begin{lemma}\label{L:proptheta}
Let $(\psi_i)_{i\in I}=\vec{\psi}\colon\vec{C}\to\vec{D}$  be an arrow of $\cS^I$, let $\bu\in\Ids(U,\module{\cdot})$. Then:
\[\psi_{\module{\bu}}\circ\pi_{\bu}^{\vec{C}}=\pi_{\bu}^{\vec{D}}\circ\Cond(\vec 
{\psi},U),\]
and
\[\Id(\Cond(\vec{\psi},U))(\theta_{\bu}^{\vec{C}})\subseteq \theta_{\bu}^{\vec {D}}.\]
\end{lemma}

\begin{proof}
Let $\vec{C}=(C_i,\gamma_{i,j})_{i\le j\text{ in }I}$, let $\vec 
{D}=(D_i,\delta_{i,j})_{i\le j\text{ in }I}$,
let $V$ be a kernel of $U$, and let $a\in\Cond_U(\vec{C},V)$. By Proposition \ref{P:CondIsFunctor}, $V$ is also a support of $\Cond(\vec{\psi},U)(a)$, and
\begin{align*}
\pi_{\bu}^{\vec{D}}(\Cond(\vec{\psi},U)(a))
&=\delta_{\module{V\cdot \bu},\module{\bu}}((\Cond(\vec{\psi},U)(a))_{V\cdot \bu})\\
&=\delta_{\module{V\cdot \bu},\module{\bu}}( \psi_{\module{V\cdot \bu}}(a_{V \cdot \bu}))\\
&=\psi_{\module{\bu}}(\gamma_{\module{V\cdot \bu},\module{\bu}}(a_{V\cdot  \bu}))\\
&=\psi_{\module{\bu}}(\pi_{\bu}^{\vec{D}}(a))
\end{align*}

The containment is an obvious consequence of the equality.
\end{proof}

\begin{definition}\label{D:quasilift}
Let $\vec{D}$ be an object of $\cS^I$. An \emph{$U$-quasi-lifting} of $\vec{D}$ in~$\cV$ is a pair $(\tau,T)$, where $T\in\cV$ and $\tau\colon\Conc T\to \Cond (\vec{D},U)$ is a $(\vee,0)$-homomorphism such that $\upw\alpha_{\bu}\to\upw\theta_{\bu}^{\vec{D}}$, $\beta\mapsto\Id(\tau)(\beta)\vee\theta_{\bu}^{\vec{D}}$ is an isomorphism, for all ${\bu}\in \Ide(U,\module{\cdot})$, where $\alpha_{\bu}=\bigvee\setm{\beta\in\Conc T}{\tau(\beta)\in\theta_{\bu}^{\vec{D}}}$.
\end{definition}

Observe that in the definition above we use the identification of $\Con T$ with the ideal lattice of $\Conc T$. We shall now extend Definition~\ref{D:quasilift} from \emph{objects} of $\cS^I$ to \emph{diagrams} of $\cS^I$.

\begin{definition}
Let~$J$ be a category and let $\cD\colon J\to\cS^I$ be a functor. An \emph{$U$-quasi-lifting} of~$\cD$ in~$\cV$ is a pair $(\vec{\tau},\cJ)$, where $\cJ\colon J\to \cV$ is a functor and $\vec{\tau}=(\tau^j)_{j\in\Ob J}\colon\Conc\circ\cJ\to\Cond(\cD(-),U)$ is a natural transformation, such that $(\tau^j,\cJ(j))$ is a $U$-quasi-lifting of $\cD(j)$ for all $j\in\Ob J$.
\end{definition}

The two following lemmas are obvious.

\begin{lemma}
Let $\vec{D}$ be an object of $\cS^I$, let $T\in\cV$, and let $\tau\colon\Conc T \to \Cond(\vec{D},U)$ be an isomorphism. Then $(\tau,T)$ is a $U$-quasi-lifting of $\vec{D}$.
\end{lemma}

\begin{lemma}
Let~$J$ be a category, let $\cD\colon J\to\cS^I$ be a functor, let $\cJ\colon J\to \cV$ be a functor, and let $\tau=(\tau^j)_{j\in\Ob J}\colon\Conc\circ\cJ\to \Cond(\cD(-),U)$ be a natural isomorphism. Then $(\tau,\cJ)$ is an $U$-quasi-lifting of~$\cD$.
\end{lemma}

The following lemma expresses a commutation property between the condensate functor $\Cond$ and the $\Conc$ functor.

\begin{lemma}\label{L:liftingimplyquasilifting}
Let $\vec{A}=(A_i,f_{i,j})_{i\le j\text{ in }I}$ be an object of $\cV^I$,
let $\vec{D}=\Conc\vec{A}=(\Conc A_i,\Conc f_{i,j})_{i\le j\text{ in }I}$, let $p_{\bu}\colon\Cond(\vec{A},U)\tosurj A_{\module{{\bu}}}$ be the canonical projection, for all ${\bu}\in\Ids(U,\module{\cdot})$. Put:
\begin{align*}
\tau\colon\Conc \Cond(\vec{A},U) &\to\Cond(\vec{D},U) \\
\beta&\mapsto ((\Conc p_{v})(\beta))_{{v}\in U}
\end{align*}
Then $(\tau,\Cond(\vec{A},U))$ is an $U$-quasi-lifting of $\vec{D}$.
\end{lemma}

\begin{proof}
Denote by $\pi_{\bu}\colon\Cond(\vec{D},U)\to D_{\module{{\bu}}}$ the canonical projection, and put $\theta_{\bu}=\theta_{\bu}^{\vec{D}}$ for all $\bu\in\Ids(U,\module{\cdot})$. Let $x,y\in\Cond(\vec{A},U)$ and put $\beta=\Theta_{\Cond(\vec{A},U)}(x,y)$. Then $\tau(\beta)=(\Theta_{A_{\module{u}}} (x_u,y_u))_{u\in U}$. Let $V$ be a common support of $x$ and $y$. For every $u\in U$,
\begin{align*}
\Theta_{A_{\module{u}}} (x_u,y_u)&=\Theta_{A_{\module{u}}}(f_{\module{V\cdot u},\module{u}}(x_{V\cdot u}),f_{\module{V\cdot u},\module{u}}(y_{V\cdot u}))\\
&=\Conc(f_{\module{V\cdot u},\module{u}})(\Theta_{A_{\module{V\cdot u}}}(x_{V\cdot u},y_{V\cdot u}) )
\end{align*}
So $V$ is a support of $\tau(\beta)$. It follows that $\tau$ takes, indeed, its values in $\Cond(\vec{D},U)$. Furthermore, for $x$, $y$, $V$, and $\beta$ as above,
\begin{align*}
\pi_{\bu}(\tau(\beta))&=\Conc(f_{\module{V\cdot\bu},\module{\bu}})(\tau(\beta)_{V\cdot\bu})\\
&=\Conc(f_{\module{V\cdot\bu},\module{\bu}})(\Theta_{A_{\module{V\cdot \bu}}}(x_{V\cdot \bu},y_{V\cdot \bu}) )\\
&=\Theta_{A_{\module{\bu}}}(p_{\bu}(x),p_{\bu}(y)),
\end{align*}
so $\pi_{\bu}\circ\tau=\Conc p_{\bu}$ for all $\bu\in\Ids(U,\module{\cdot})$.

Let ${\bu}\in\Ids(U,\module{\cdot})$ and put $\alpha_{\bu}=\bigvee\setm{\beta\in\Conc(\Cond(\vec{A},U))}{\tau(\beta)\in\theta_{\bu}}$. The following equivalences hold, for every $\beta\in\Conc\Cond(\vec{A},U)$:
\begin{align*}
\beta\subseteq\ker p_{\bu} &\Longleftrightarrow \Con_c(p_{\bu})(\beta)=0\\
&\Longleftrightarrow\pi_{\bu}\circ \tau(\beta)=0\\
&\Longleftrightarrow\tau(\beta)\in\theta_{\bu}\\
&\Longleftrightarrow\beta\subseteq\alpha_{\bu},
\end{align*}
thus $\alpha_{\bu}=\ker p_{\bu}$. Let $\tau_{\bu}\colon\upw\alpha_{\bu}\to\upw\theta_{\bu}$ be the map defined by $\tau_{\bu}(\beta)=(\Id\tau)(\beta)\vee\theta_{\bu}$ for all $\beta\in\Con\Cond(\vec{A},U)$ containing $\alpha_{\bu}$. As $(\Id\pi_{\bu})(\theta_{\bu})=0$, the following diagram is commutative:
\[
\xymatrix{
\upw\theta_{\bu}  \ar[rrd]^{\Id(\pi_{\bu})}&\\
\upw\alpha_{\bu} \ar[u]_{\tau_{\bu}} \ar[rr]_-{\Con(p_{\bu})} & & \Con(A_{\module{{\bu}}})
}
\]
As both $(\Id\pi_{\bu})\res\upw\theta_{\bu}$ and $(\Con p_{\bu})\res\upw\alpha_{\bu}$ are isomorphisms, so is $\tau_{\bu}$.
\end{proof}

\begin{lemma}\label{L:liftingimplyquasiliftingext}
Let~$J$ be a category, let $\cA\colon J\rightarrow\cV^I$ be a functor, put $\cD=\Conc\circ\cA$. Let $\tau^j\colon\Conc(\Cond(\cA(j),U))\to\Cond(\Conc(\cA(j)),U)$ be the maps defined in Lemma~\textup{\ref{L:liftingimplyquasilifting}}, for all $j\in J$. Let $\tau=(\tau^j)_{j\in\Ob J}$. Then $(\tau,\Cond(\cA(-),U))$ is a $U$-quasi-lifting of~$\cD$.
\end{lemma}

\begin{proof}
By Lemma~\ref{L:liftingimplyquasilifting}, $(\tau^j,\Cond(\cA(j),U))$ is a $U$-quasi-lifting of~$\cD(j)$, for all $j\in\Ob J$.

Let $\cA(j)=(A^j_i,t_{i,i'}^j)_{i\le i'\text{ in }I}$, for all $j\in\Ob J$. Let $f\colon j\to k$ be an arrow of~$J$, let $\cA(f)=(a^f_i)_{i\in I}$, let $p_u^k\colon\Cond(\cA(k),U)\tosurj A_{\module{u}}^k$ be the canonical projection, for all $u\in U$. Let $x,y\in\Cond(\cA(j),U)$. Then:
\begin{align*}
&\Cond\big(\Conc\cA(f),U\big)\Big(\tau^j\big(\Theta_{\Cond(\cA(j),U)}(x,y)\big)\Big)\\
&=\Cond\big(\Conc\cA(f),U\big)\Big( \big(\Theta_{A_{\module{u}}^j}(x_u,y_u)\big)_{u\in U} \Big)\\
&=\Big( \Conc(a^f_{\module{u}}) \big(\Theta_{A_{\module{u}}^j}(x_u,y_u)\big) \Big)_{u\in U}\\
&=\Big( \Theta_{A_{\module{u}}^k}\big(a^f_{\module{u}}(x_u),a^f_{\module{u}}(y_u)\big) \Big)_{u\in U}\\
&=\Big( \Theta_{A_{\module{u}}^k}\big(p_u^k(\Cond(\cA(f),U)(x)),p_u^k(\Cond(\cA(f),U)(y))\big)\Big)_{u\in U}\\
&=\Big( \Conc(p_u^k)\big(\Theta_{\Cond(\cA(k),U)}\big(\Cond(\cA(f),U)(x),\Cond(\cA(f),U)(y)\big)\big)\Big)_{u\in U}\\
&=\tau^k\Big( \Theta_{\Cond(\cA(k),U)}\big(\Cond(\cA(f),U)(x),\Cond(\cA(f),U)(y)\big)\Big)\\
&=\tau^k\Big(  \Conc\big(\Cond(\cA(f),U)\big) \big(\Theta_{\Cond(\cA(j),U)}(x,y)\big)\Big)
\end{align*}
So the following diagram is commutative:
\[ 
\begin{CD}
\Cond(\Conc\cA(j),U)  @>\Cond(\Conc\cA(f),U)>>   \Cond(\Conc\cA(k),U)\\
@A\tau^jAA                            @AA\tau^kA\\
\Conc\Cond(\cA(j),U) @>\Conc\Cond(\cA(f),U)>>   \Conc\Cond(\cA(k),U) 
\end{CD}
\]
This concludes the proof.
\end{proof}

\begin{theorem}\label{T:mainlifting}
Let~$J$ be a small category, suppose that~$\cV$ is closed under homomorphic images, and is $\Big(\Ide(U,\module{\cdot})^=,J,(\kappa_{\module{\bu}})_{\bu\in\Ide(U,\module{\cdot})^=}\Big)$-L\"owenheim-Skolem. Let $\cD\colon J\to\cS^I$ be a functor, let $(\vec{\tau},\cA)$ be a $U$-quasi-lifting of~$\cD$ in~$\cV$. Let $\cD(j)=\vec{D}^j=\Big(D_i^j,\phi_{i,i'}^j\Big)_{i\le i'\text{ in }I}$, for all $j\in\Ob J$, let $\cD(f)=\vec{\psi}^f=(\psi_i^f)_{i\in I}$, for all $f\in\Mor J$. If $\sum_{j\in \Ob J}\card{D_i^j}<\kappa_i$, for all $i\in I$, then there exists a lifting in~$\cV$ of the diagram $\widehat{\cD}\colon I\times J\to\cS$, associated to~$\cD$ \textup(cf. Section~\textup{\ref{S:Basic}}\textup).
\end{theorem}

\begin{proof}
Let $\theta_{\bu}^j=\theta_{\bu}^{\cD(j)}$ as defined in Proposition~\ref{P:thetaisanideal}, let $\alpha_{\bu}^j=\bigvee\setm{\beta\in\Conc\cA(j)}{\tau^j(\beta)\le\theta_{\bu}^j}$, let $\tau_{\bu}^j\colon\upw\alpha_{\bu}^j\to\upw\theta_{\bu}^j$, $\beta\mapsto\Id(\tau^j)(\beta)\vee\theta_{\bu}^j$, as in Definition~\ref{D:quasilift}, let $p_{\bu}^j\colon\cA(j)\tosurj\cA(j)/\alpha_{\bu}^j$ the canonical projection, and let $\pi_{\bu}^j\colon\Cond(\cD(j),U)\tosurj D_{\module{{\bu}}}^j$, the canonical projection as defined in Lemma~\ref{L:prop_de_base_cond}(7), for all ${\bu}\in\Ide(U,\module{\cdot})$ and all $j\in\Ob J$.
The map $\chi_{\bu}^j=\Con(p_{\bu}^j)\circ(\tau_{\bu}^j)^{-1}\circ(\Id(\pi_{\bu}^j)\res\upw\theta_{\bu}^j)^{-1}$ is an isomorphism.
\[
\xymatrix{
\Id(D_{\module{{\bu}}}^j) \ar@/^2pc/[rrrrrr]^{\chi_{\bu}^j} & &  \upw\theta_{\bu}^j \ar[ll]^{\Id \pi_{\bu}^j} & & \upw\alpha_{\bu}^j\ar[ll]^{\tau_{\bu}^j}  \ar[rr]_-{\Con p_{\bu}^j}  & &  \Con(\cA(j)/\alpha_{\bu}^j)
}
\]
Moreover $\sum_{j\in \Ob J}\card \Con_c(\cA(j)/\alpha_{\bu}^j) = \sum_{j\in \Ob J}\card D_{\module{{\bu}}}^j<\kappa_{\module{\bu}}$, for all ${\bu}\in \Ide(U,\module{\cdot})^=$. So there exists a family $(B_{\bu}^j)_{\bu\in \Ide(U,\module{\cdot})^=}^{j\in\Ob J}$ of algebras such that:
\begin{enumerate}
\item The algebra $B_{\bu}^j$ is a subalgebra of $\cA(j)$.
\item The algebra $B_{\bu}^j/\alpha_{\bu}^j$ belongs to~$\cV$.
\item The containment $B_{\bu}^j\subseteq B_{\bv}^j$ holds.
\item The containment $\cA(f)(B_{\bu}^j)\subseteq B_{\bu}^k$ holds.
\item The morphism $\Con(q_{\bu}^j)$ is an isomorphism, where $q_{\bu}^j\colon B_{\bu}^j/\alpha_{\bu}^j\toinj \cA(j)/\alpha_{\bu}^j$ denotes the canonical embedding.
\item The inequality $\sum_{l\in \Ob J}\card{B_{\bu}^l}<\kappa_{\module{{\bu}}}$ holds.
\end{enumerate}
for all $\bu\le\bv$ in $\Ide(U,\module{\cdot})^=$ and for every morphism $f\colon j\to k$ in~$J$. Moreover, we can extend this family to $\Ide(U,\module{\cdot})$, by $B_{\bu}^j=\cA(j)$, the statements (1)--(5) hold for all $\bu\le\bv$ in $\Ide(U,\module{\cdot})$, and for every morphism $f\colon j\to k$ in~$J$.

Put:
\begin{align*}
F\colon\Ide(U,\module{\cdot})&\to\Pow(U)\\
\bu&\mapsto \bigcup\setm{\supp \tau^j(\Theta_{\cA(j)}(x,y))}{j\in\Ob J\text{ and }x,y\in B_{\bu}^j}
\end{align*}
so $F(\bu)<\kappa_{\module{\bu}}$ for all $\bu\in \Ide(U,\module{\cdot})^=$. As $(U,\module{\cdot})$ is $\vec{\kappa}$-compatible there exists an order-preserving map $\sigma\colon I\to \Ide(U,\module{\cdot})$ such that:
\begin{enumerate}
\item The equality $\module{\sigma(i)}=i$ holds for all $i\in I$.
\item The equality $V\cdot\sigma(i)=V\cdot\sigma(i')$ holds for any $i\le i'$ in $I$ and any kernel $V$ of $U$ contained in $F(\sigma(i))$.
\end{enumerate}
Let $i\in I$ and $j\in\Ob J$. The map $\xi_i^j=(\Con(q_{\sigma(i)}^j))^{-1}\circ \chi_{\sigma(i)}$ is an isomorphism, and the algebra $\cB(i,j)=B_{\sigma(i)}^j/\alpha_{\sigma(i)}^j\in\cV$ belongs to~$\cV$.
\[
\xymatrix @C=0.7in{
\Id(D_i^j) \ar@/^2pc/[rrr]^{\chi_{\sigma(i)}^j} \ar@/^-2pc/[drrr]_{\xi_i^j} &   \upw\theta_{\sigma(i)}^j \ar[l]^-{\Id(\pi_{\sigma(i)}^j)} &  \upw\alpha_{\sigma(i)}^j\ar[l]^{\tau_{\sigma(i)}^j}  \ar[r]_-{\Con(p_{\sigma(i)}^j)}  & \Con(\cA(j)/\alpha_{\sigma(i)}^j)\\
& & & \Con(\cB(i,j)) \ar[u]_-{\Con(q_{\sigma(i)}^j)}
}
\]

Let $i\le i'$ in $I$, let $j\in\Ob J$, let $x,y\in B_{\sigma(i)}^j$, let $\beta=\Theta_{\cA(j)}(x,y)$. The following equalities hold:
\begin{align}
\Con(p_{\sigma(i)}^j)(\beta\vee\alpha_{\sigma(i)}^j)&=\Theta_{\cA(j)/\alpha_{\sigma(i)}^j}(x/\alpha_{\sigma(i)}^j,y/\alpha_{\sigma(i)}^j) \notag\\
&=\Con(q_{\sigma(i)}^j)(\Theta_{\cB(i,j)}(x/\alpha_{\sigma(i)}^j,y/\alpha_{\sigma(i)}^j)).  \label{E:MainThEq0a}
\end{align}
similarly:
\begin{equation}\label{E:MainThEq0b}
\Con(p_{\sigma(i')}^j)(\beta)=\Con(q_{\sigma(i')}^j)(\Theta_{\cB(i',j)}(x/\alpha_{\sigma(i')}^j,y/\alpha_{\sigma(i')}^j)).
\end{equation}
Moreover, set $V=\supp(\tau^j(\beta))$. Then $V\subseteq F(\sigma(i))$, so $V\cdot\sigma(i)= V\cdot\sigma(i')$ and so:
\begin{align*}
\pi_{\sigma(i')}^j(\tau^j(\beta)) &= \phi_{\module{V\cdot\sigma(i')},\module{\sigma(i')}}(\tau^j(\beta)_{V\cdot\sigma(i')}) & & \text{by Lemma~\ref{L:prop_de_base_cond}(7)}\\
&= \phi_{\module{V\cdot\sigma(i)},i'}(\tau^j(\beta)_{V\cdot\sigma(i)}) & & \text{as $V\cdot\sigma(i')=V\cdot\sigma(i)$ and $\module{\sigma(i')}=i'$}\\
&= \phi_{i,i'}\circ\phi_{\module{V\cdot\sigma(i)},i}(\tau^j(\beta)_{V\cdot\sigma(i)})\\
&= \phi_{i,i'}\circ\pi_{\sigma(i)}^j(\tau^j(\beta)) & &\text{by Lemma~\ref{L:prop_de_base_cond}(7).}
\end{align*}
So:
\begin{equation}\label{E:MainThEq1}
\Id(\phi_{i,i'}^j)\circ\Id(\pi_{\sigma(i)}^j)\circ\Id(\tau^j)(\beta)= \Id(\pi_{\sigma(i')}^j)\circ\Id(\tau^j)(\beta)
\end{equation}
As $\alpha_{\bu}^j=\bigvee\setm{\beta\in\Conc\cA(j)}{\tau^j(\beta)\le\theta_{\bu}^j}$, we have $\Id(\tau^j)(\alpha_{\bu}^j)\le\theta_{\bu}^j$. Thus:
\begin{equation}\label{E:MainThEq2}
\tau_{\bu}^j(\beta\vee\alpha_{\bu}^j)=\Id(\tau^j)(\beta\vee\alpha_{\bu}^j)\vee\theta_{\bu}^j=\Id(\tau^j)(\beta)\vee\theta_{\bu}^j,\quad\text{for all $\bu\in\Ide(U,\module{\cdot})$.}
\end{equation}
As $\Id(\pi_{\bu}^j)(\theta_{\bu}^j)=0$, the following equation holds:
\begin{equation}\label{E:MainThEq3}
\Id(\pi_{\bu}^j)\circ\tau_{\bu}^j(\beta\vee\alpha_{\bu}^j)=\Id(\pi_{\bu}^j)\circ\Id(\tau^j)(\beta),\quad\text{for all $\bu\in\Ide(U,\module{\cdot})$.}
\end{equation}
So:
\begin{align}
\Id(\phi_{i,i'}^j)\circ\Id(\pi_{\sigma(i)}^j) \circ\tau_{\sigma(i)}^j(\beta\vee\alpha_{\sigma(i)}^j)
   &=\Id(\phi_{i,i'}^j)\circ\Id(\pi_{\sigma(i)}^j)\circ\Id(\tau^j)(\beta)    & & \text{by \eqref{E:MainThEq3}}\notag\\
   &=\Id(\pi_{\sigma(i')}^j)\circ\Id(\tau^j)(\beta) & & \text{by \eqref{E:MainThEq1}}\notag
\end{align}
and so, by \eqref{E:MainThEq3}, the following equality holds
\begin{equation}\label{E:MainThEq4}
\Id(\phi_{i,i'}^j)\circ\Id(\pi_{\sigma(i)}^j) \circ\tau_{\sigma(i)}^j(\beta\vee\alpha_{\sigma(i)}^j)=\Id(\pi_{\sigma(i')}^j)\circ\Id(\tau^j_{\sigma(i')})(\beta\vee\alpha_{\sigma(i')}^j) 
\end{equation}
thus:
\begin{align*}
&\xi_{i'}^j\circ\Id(\phi_{i,i'}^j)\circ(\xi_i^j)^{-1}(\Theta_{\cB(i,j)}(x/\alpha_{\sigma(i)}^j,y/\alpha_{\sigma(i)}^j))\\
&=\xi_{i'}^j\circ\Id(\phi_{i,i'}^j)\circ\Id(\pi_{\sigma(i)}^j)\circ\tau_{\sigma(i)}^j\circ(\Con(p_{\sigma(i)}^j)\res\upw\alpha_{\sigma(i)}^j)^{-1}\\
&\quad\quad\quad \circ\Con(q_{\sigma(i)}^j) (\Theta_{\cB(i,j)}(x/\alpha_{\sigma(i)}^j,y/\alpha_{\sigma(i)}^j))\\
&=\xi_{i'}^j\circ\Id(\phi_{i,i'}^j)\circ\Id(\pi_{\sigma(i)}^j)\circ\tau_{\sigma(i)}^j(\beta\vee\alpha_{\sigma(i)}^j) & & \text{by \eqref{E:MainThEq0a}}\\
&=\xi_{i'}^j\circ\Id(\pi_{\sigma(i')}^j)\circ\tau_{\sigma(i')}^j(\beta\vee\alpha_{\sigma(i')}^j) & & \text{by \eqref{E:MainThEq4}}\\
&=(\Con q_{\sigma(i')}^j)^{-1}\circ\Con(p_{\sigma(i')}^j)\circ(\tau_{\sigma(i')}^j)^{-1}\\
&\quad\quad\quad \circ\Id(\pi_{\sigma(i')}^j\res\upw\theta_{\sigma(i')}^j)^{-1}\circ\Id(\pi_{\sigma(i')}^j)\circ\tau_{\sigma(i')}^j(\beta\vee\alpha_{\sigma(i')}^j)\\
&=(\Con q_{\sigma(i')}^j)^{-1}\circ\Con(p_{\sigma(i')}^j)(\beta\vee\alpha_{\sigma(i')}^j)\\
&=\Theta_{\cB(i',j)}(x/\alpha_{\sigma(i')}^j,y/\alpha_{\sigma(i')}^j). & & \text{by \eqref{E:MainThEq0b}}
\end{align*}

It follows that the following morphism is well-defined:
\begin{align*}
g_{i,i'}^j\colon \cB(i,j) &\to \cB(i',j)\\
x/\alpha_{\sigma(i)}^j &\mapsto x/\alpha_{\sigma(i')}^j
\end{align*}
and the following diagram is commutative:
\begin{equation}\label{CD:MainThCD1}
\begin{CD}
\Id(D_{i'}^j) @>\xi_{i'}^j>> \Con(\cB(i',j))\\
@A\Id(\phi_{i,i'}^j)AA @AA\Con(g_{i,i'}^j)A \\
\Id(D_{i}^j) @>\xi_{i}^j>> \Con(\cB(i,j)) \\
\end{CD}
\end{equation}

Let $f\colon j\to k$ be an arrow of $J$, let $i\in I$, and put $\bu=\sigma(i)$. As $(\vec{\tau},\cA)$ is a $U$-quasi-lifting of~$\cD$, the following diagram is commutative:
\begin{equation}\label{CD:MainThCD2}
\begin{CD}
\Id(D_i^k) @<\Id(\pi_{\bu}^k)<< \Id(\Cond(\cD(k),U)) @<\Id(\tau^k)<< \Con(\cA(k)) \\
@A\Id(\psi_i^f)AA @A\Cond(\cD(f),U)AA @A\Con(\cA(f))AA \\
\Id(D_i^j) @<\Id(\pi_{\bu}^j)<< \Id(\Cond(\cD(j),U)) @<\Id(\tau^j)<< \Con(\cA(j)) \\
\end{CD}
\end{equation}
Let $\beta\in\Conc\cA(j)$ such that $\tau^j(\beta)\in\theta_{\bu}^k$. Thus $\pi_{\bu}^j(\tau^j(\beta))=0$, so:
\[
0=\psi_i^f(\pi_{\bu}^j(\tau^j(\beta)))=\pi_{\bu}^k\Big(\tau^k\big(\Conc(\cA(f))(\beta)\big)\Big)
\]
and so $\Con(\cA(f))(\beta)\le\alpha_{\bu}^k$. Thus:
\[\Con(\cA(f))(\alpha_{\bu}^j)=\Con(\cA(f))\big(\bigvee\setm{\beta\in\Conc(\cA(j))}{\tau^j(\beta)\in\theta_{\bu}^k}\big)\le\alpha_{\bu}^k\]

So the following morphism is well-defined:
\begin{align*}
\tilde{f}_i\colon \cA(j)/\alpha_{\bu}^j &\to \cA(k)/\alpha_{\bu}^k\\
x/\alpha_{\bu}^j &\mapsto \cA(f)(x)/\alpha_{\bu}^k
\end{align*}
and the following diagram is commutative:
\begin{equation}\label{CD:MainThCD3}
\begin{CD}
\cA(k) @>p_{\sigma(i)}^k>> \cA(k)/\alpha_{\sigma(i)}^k   @<q_{\sigma(i)}^k<< \cB(i,k)\\
@A\cA(f)AA            @A{\tilde{f}_i}AA           @AA{\tilde{f}_i\res\cB(i,j)}A\\
\cA(j) @>p_{\sigma(i)}^j>> \cA(j)/\alpha_{\sigma(i)}^j   @<q_{\sigma(i)}^j<< \cB(i,j)
\end{CD}
\end{equation}
Combining the commutative diagrams \eqref{CD:MainThCD3} and \eqref{CD:MainThCD4} together with the definitions of $\xi_{i}^k$ and $\xi_{i}^j$, we obtain the commutativity of the following diagram:
\begin{equation}\label{CD:MainThCD4}
\begin{CD}
\Id(D_{i}^k) @>\xi_{i}^k>> \Con(\cB(i,k))\\
@A\Id(\psi_i^f)AA @AA\Con(\tilde{f}_i\res\cB(i,j))A \\
\Id(D_{i}^j) @>\xi_{i}^j>> \Con(\cB(i,j)) \\
\end{CD}
\end{equation}

For $i\le i'$ in $I$ and $f\colon j\to k$ in $J$, put
\begin{align*}
\cB(i\to i',f\colon j\to k)\colon \cB(i,j) &\to \cB(i',k)\\
x/\alpha_{\sigma(i)}^j &\mapsto \cA(f)(x)/\alpha_{\sigma(i')}^k.
\end{align*}
Let $i''\ge i'$ in $I$ and $f'\colon k\to k'$ in $J$, then:
\begin{align*}
\cB(i'\to i'',f')\circ\cB(i\to i',f)(x/\alpha_{\sigma(i)}^j)&=\cB(i'\to i'',f')(\cA(f)(x)/\alpha_{\sigma(i')}^k)\\
&=\cA(f')(\cA(f)(x))/\alpha_{\sigma(i')}^{k'}\\
&=\cA(f'\circ f)(x)/\alpha_{\sigma(i')}^{k'}\\
&=\cB(i\to i'',f'\circ f)(x/\alpha_{\sigma(i)}^j).
\end{align*}
Thus $\cB\colon I\times J\to\cV$ is a functor. Moreover by \eqref{CD:MainThCD1} and \eqref{CD:MainThCD4} the following diagram is commutative:
\[
\begin{CD}
\Id(D_{i'}^k)         @>\xi_{i'}^k>>      \Con(\cB(i',k))\\
@A\Id(\phi_{i,i'}^k)AA   @AA\Con(g_{i,i'}^k)A\\
\Id(D_{i}^k)         @>\xi_{i'}^k>>      \Con(\cB(i,k))\\
@A{\Id(\psi_i^f)}AA                        @AA{\Con(\tilde{f}_i\res\cB(i,j)))}A \\
\Id(D_{i}^j)         @>\xi_{i}^j>>      \Con(\cB(i,j))\\
\end{CD}
\]
As $\widehat{\cD}(i\le i',f) = \phi_{i,i'}^k\circ \psi_i^f$ and $\cB(i\to i',f)=g_{i,i'}^k\circ \tilde{f}_i$, the following diagram is commutative:
\[
\begin{CD}
\Id(\widehat{\cD}(i',k))         @>\xi_{i'}^k>>      \Con(\cB(i',k))\\
@A\Id(\vec{\cD}(i\to i',f))AA  @AA\Con(\cB(i\to i',f))A\\
\Id(\widehat{\cD}(i,j))        @>\xi_{i}^j>>      \Con(\cB(i,j))
\end{CD}
\]
\end{proof}

\section{Critical points}

\begin{definition}
Let~$\cV$ be a class of algebras of the same similarity type. The \emph{congruence class} of~$\cV$ is the class of all \jzs s $S$ such there exists $A\in\cV$ such that $S$ isomorphic to $\Conc A$. We denote this class by $\Conc\cV$.
\end{definition}

\begin{definition}
Let~$\cV_1$ be a class of algebras of the same similarity type, let~$\cV_2$ be a class of algebras of the same similarity type. The \emph{critical point of~$\cV_1$ under~$\cV_2$} is:
\[\crit{\cV_1}{\cV_2}=\min\setm{\card{D}}{D\in \Conc(\cV_1) - \Conc(\cV_2)},\]
if $\Conc\cV_1\not\subseteq\Conc\cV_2$, otherwise we put $\crit{\cV_1}{\cV_2}=\infty$.

The \emph{symmetric critical point of~$\cV_1$ and~$\cV_2$} is defined as
 \[
 \critsym{\cV_1}{\cV_2}=\min\set{\crit{\cV_1}{\cV_2}, \crit{\cV_2}{\cV_1} };
 \]
it is simply called \emph{critical point} in~\cite{CLPSurv}.
\end{definition}

The following corollary shows that, for a fixed category $J$ and a tree $T$, if~$\cV_1$ and~$\cV_2$ lift the same diagrams of \jzs s, indexed by $J$, of not too large objects, then~$\cV_1$ and~$\cV_2$ lift the same diagrams of \jzs s, indexed by $T\times J$, of not too large objects. The condition $(1)$ above is automatically satisfied if $\card\sL_1\le\kappa$ and $\card\sL_2\le\lambda$.

\begin{corollary}\label{C:liftinginitialstep}
Let $\cS$ be the variety of all \jzs s, let $\sL_1$ and $\sL_2$ be similarity types, let~$\cV_1$ be a variety of $\sL_1$-algebras, let~$\cV_2$ be a variety of $\sL_2$-algebras, let $\lambda<\kappa$ be infinite cardinals, let $I$ be a well-founded tree, and let~$J$ be a small category, such that:
\begin{enumerate}
\item $\cV_1$ is locally $\le\kappa$ and $\cV_2$ is locally $\le\lambda$.
\item $\card\Mor(J) <\kappa$.
\item $\card I \le\kappa$.
\item $\card(\dnw i)<\cf\kappa$ for all $i\in I$.
\item Every functor $\cD\colon J\to\cS$ such that $\card\cD(j)\le\kappa$ for all $j\in\Ob J$, which has a lifting in~$\cV_1$, has a lifting in~$\cV_2$.
\end{enumerate}
Then every functor $\cD\colon I\times J\to\cS$ such that $\card\cD(i,j)<\kappa$ for all $i\in I$ and all $j\in\Ob J$, which has a lifting in~$\cV_1$, has a lifting in~$\cV_2$.
\end{corollary}

\begin{proof}
Let $\cD\colon I\times J\to\cS$ be a functor such that $\card\cD(i,j)<\kappa$ for all $i\in I$ and all $j\in\Ob J$, let $\cA\colon I\times J\to\cV_1$ be a lifting of~$\cD$, denote by $\alpha_{i,j}$ the identity congruence of $\cA(i,j)$, for all $i\in I$ and all $j\in\Ob J$. By using Lemma~\ref{L:LS} we can assume that:
\[\card\cA(i,j)\le \kappa + \sum_{i'\le i}\sum_{j'\in\Ob J}\sum_{f\colon j'\to j}\card\cD(i',j') \le \kappa \]
Moreover by Corollary~\ref{C:Existnormcovering} there exists a tight $\kappa$-compatible norm-covering $(U,\module{\cdot})$ of $I$ such that $\card U\le\kappa$. As seen in Section~\ref{S:Basic}, the functor $\cA$ corresponds to a functor $\widetilde{\cA}\colon J\to\cV_1^I$ and the functor~$\cD$ corresponds to a functor $\widetilde{\cD}\colon J\to\cS^I$. Lemma~\ref{L:liftingimplyquasiliftingext} implies that there exists $\tau=(\tau^j)_{j\in \Ob J}$ such that $(\tau,\Cond(\widetilde{\cA}(-),U))$ is a $U$-quasi-lifting of $\widetilde{\cD}$, and:
\[\card \Cond(\widetilde{\cA}(j),U) \le \sum_{V\in\Powf{U}}\card \prod_{u\in V} \cA(\module{u},j)\le \sum_{V\in\Powf{U}}\kappa\le\kappa,\]
for all $j\in\Ob J$. So there exists a lifting of $\Conc\Cond(\widetilde{\cA}(-),U))$ in~$\cV_2$, and so there exists a $U$-quasi-lifting $\cB\colon J\to\cV_2$ of $\widetilde{\cD}$ in~$\cV_2$. By Lemma~\ref{L:LS2},~$\cV_2$ is $(\Ide(U,\module{\cdot})^=,J,(\kappa)_{\bu\in\Ide(U,\module{\cdot})^=} )$-L\"owenheim-Skolem, so, by Theorem \ref{T:mainlifting},~$\cD$ has a lifting in~$\cV_2$.
\end{proof}

Using a simple induction argument, we obtain the following corollary.

\begin{corollary}\label{C:lifting-with-nth-succesor-cardinal}
Let $\cS$ be the variety of all \jzs s, let $\sL_1$ and $\sL_2$ be similarity types, let~$\cV_1$ be a variety of $\sL_1$-algebras, let~$\cV_2$ be a variety of $\sL_2$-algebras, let $\kappa$ be an infinite cardinal, let $I_1,I_2,\dots,I_n$ be well-founded trees, and let~$J$ be a category, such that:
\begin{enumerate}
\item $\cV_1$ is locally $\le\kappa^+$ and $\cV_2$ is locally $\le\kappa$.
\item $\card I_1 + \card I_2 + \dots + \card I_{n-1} + \card\Mor J\le\kappa$.
\item $\card I_n \le\kappa^+$.
\item $\card \dnw i\le\kappa$ for each $i\in I_n$.
\item Every diagram of \jzs s $\cD\colon J\to\cS$, such that $\card\cD(j)\le\kappa^{+n}$, which has a lifting in~$\cV_1$ has a lifting in~$\cV_2$.
\end{enumerate}
Then every diagram of \jzs s $\cD\colon I_1\times I_2\times \dots\times I_n\times J\to\cS$, such that $\card\cD(i_1,i_2,\dots,i_n,j)\le\kappa$ for all $(i_1,i_2,\dots,i_n,j) \in I_1\times I_2\times\dots\times I_n\times\Ob J$, which has a lifting in~$\cV_1$ has a lifting in~$\cV_2$.
\end{corollary}

The following corollary is similar to Corollary~\ref{C:liftinginitialstep}. It shows that with finitely generated congruence-distributive varieties of algebras we can go one step further.

\begin{corollary}\label{C:lifting-inFGCDvariety}
Let $\cS$ be the variety of all \jzs s, let $\sL_1$ and $\sL_2$ be similarity types, let~$\cV_1$ be a variety of $\sL_1$-algebras, let~$\cV_2$ be a finitely generated congruence-distributive variety of $\sL_2$-algebras, let $I$ be a lower finite tree, and let~$J$ be a finite poset, such that:
\begin{enumerate}
\item $\cV_1$ is locally $\le\aleph_0$.
\item $\card I\le\aleph_0$.
\item Every functor $\cD\colon J\to\cS$ such that $\card\cD(j)\le\aleph_0$ for all $j\in J$, which has a lifting in~$\cV_1$ has a lifting in~$\cV_2$.
\end{enumerate}
Then every functor $\cD\colon I\times J\to\cS$, such that $\cD(i,j)$ is finite for all $(i,j)\in I \times J$, which has a lifting in~$\cV_1$ has a lifting in~$\cV_2$.
\end{corollary}

\begin{proof}
Let $\cD\colon I\times J\to\cS$ be a functor such that $\cD(i,j)$ is finite for all $(i,j)\in I\times J$. Let $\cA\colon I\times J\to\cV_1$ be a lifting of~$\cD$. Denote by $\alpha_{i,j}$ the identity congruence of $\cA(i,j)$, for all $(i,j)\in I\times J$. By using Lemma~\ref{L:LS}, we can assume that:
\[\card\cA(i,j)\le \aleph_0 + \sum_{i'\le i}\sum_{j'\le j}\card\cD(i',j') \le \aleph_0 \]

Moreover, by Corollary~\ref{C:Existnormcovering}, there exists a tight $\aleph_0$-compatible norm-covering $(U,\module{\cdot})$ of $I$ such that $\card U\le\aleph_0$. The functor $\cA$ corresponds to a functor $\widetilde{\cA}\colon J\to\cV^I$ and the functor~$\cD$ corresponds to a functor $\widetilde{\cD}\colon J\to\cV^I$. Lemma~\ref{L:liftingimplyquasiliftingext} implies that there exists $\tau=(\tau^j_{j\in J})$ such that $(\tau,\Cond(\widetilde{\cA}(-),U))$ is a $U$-quasi-lifting of $\widetilde{\cD}$, and:
\[\card \Cond(\widetilde{\cA}(j),U) \le \sum_{V\in\Powf{U}}\card \prod_{u\in V} \cA(\module{u},j)\le \sum_{V\in\Powf{U}}\aleph_0=\aleph_0,\ \text{for all $j\in J$.}\]

Lemma~\ref{L:LS-FGCDV} shows that~$\cV_2$ is $(\Ide(U,\module{\cdot})^=,J,\aleph_0)$-L\"owenheim-Skolem. By Theorem~\ref{T:mainlifting},~$\cD$ has a lifting in~$\cV_2$.
\end{proof}

Combining Corollary~\ref{C:lifting-with-nth-succesor-cardinal} and Corollary~\ref{C:lifting-inFGCDvariety} gives us the following corollary. This result is similar to Corollary~\ref{C:lifting-with-nth-succesor-cardinal}, but it involves diagrams of finite \jzs s. This makes it possible to give a bound on the critical point, in case we can find a finite diagram of finite \jzs s, indexed by some Boolean algebra, with a lifting in the first variety but with no lifting in the second one.

\begin{corollary}\label{C:lifting-inFGCDvariety-with-alephn}
Let $\cS$ be the variety of all \jzs s, let $\sL_1$ and $\sL_2$ be similarity types, let~$\cV_1$ be a variety of $\sL_1$-algebras locally $\le\aleph_0$, let~$\cV_2$ be a finitely generated congruence-distributive variety of $\sL_2$-algebras, and let $I_1,I_2,\dots, I_n$ be finite trees, let $I_{n+1}$ be a lower finite countable tree. If $\crit{\cV_1}{\cV_2}>\aleph_n$, then every functor $\cD\colon I_1\times I_2\times \dots\times I_{n+1}\to\cS$ such that $\cD(i_1,i_2,\dots,i_{n+1})$ is finite for all $(i_1,i_2,\dots,i_{n+1})\in I_1\times I_2\times\dots \times I_{n+1}$, which has a lifting in~$\cV_1$, has a lifting in~$\cV_2$.
\end{corollary}

The following corollary is a variant of Corollary~\ref{C:liftinginitialstep} that involves a class of \jzs s and a variety of algebras.

\begin{corollary}
Let $\cS$ be a class of \jzs s \textup(resp., \jzus s\textup) closed under finite products and directed unions \textup(resp., directed unions preserving $0$ and $1$\textup), let $\sL$ be a similarity type, let~$\cV$ be a variety of $\sL$-algebras, let $\lambda<\kappa$ be infinite cardinals, let $I$ be a well-founded tree, and let~$J$ be a category, such that:
\begin{enumerate}
\item $\lambda+\card\Mor(J) <\kappa$.
\item $\card I\le\kappa$.
\item $\card (\dnw i)<\cf\kappa$ for all $i\in I$.
\item Every diagram of \jzs s \textup(resp., \jzus s\textup) $\cD\colon J\to\cS$ such that $\card\cD(j)\le\kappa$ for all $j\in\Ob J$, has a lifting in~$\cV$.
\end{enumerate}
Then every functor $\cD\colon I\times J\to\cS$ such that $\card\cD(i,j)<\kappa$ for all $i\in I$ and all $j\in\Ob J$, has a lifting in~$\cV$.
\end{corollary}

\begin{proof}
By Corollary~\ref{C:Existnormcovering} there exists a tight $\kappa$-compatible norm-covering $(U,\module{\cdot})$ of $I$ such that $\card U\le\kappa$. Let $\cD\colon I\times J\to\cS$ be a diagram of \jzs s (resp., \jzus s) such that $\card\cD(i,j)<\kappa$ for all $i\in I$ and all $j\in\Ob J$. This functor corresponds to a functor $\widetilde{\cD}\colon J\to\cS^I$. But: 
\[\card\Cond(\widetilde{\cD}(j),U) \le \sum_{V\in\Powf{U}}\card\prod_{u\in V}\cD(\module{u},j)\le \sum_{V\in\Powf{U}}\kappa\le\kappa,\]
for all $j\in\Ob J$. Moreover $\Cond(\widetilde{\cD}(-),U)$ is a diagram of \jzs s (resp., \jzus s) of $\cS$. So $\Cond(\widetilde{\cD}(-),U)$ has a lifting $\cA\colon J\to\cV$, and, by Lemma~\ref{L:liftingimplyquasilifting}, $\cA\colon J\to\cV$ is a $U$-quasi-lifting of $\widetilde{\cD}$. Moreover, by Lemma~\ref{L:LS2},~$\cV$ is $(\Ide(U,\module{\cdot})^=,J,(\kappa)_{\bu\in\Ide(U,\module{\cdot})^=})$-L\"owenheim-Skolem. Hence, by Theorem \ref{T:mainlifting},~$\cD$ has a lifting in~$\cV$.
\end{proof}

By an easy induction argument we obtain the following:

\begin{corollary}
Let $\cS$ be a class of \jzs s \textup(resp., \jzus s\textup) closed under finite products and directed unions \textup(resp., directed unions preserving $0$ and $1$\textup), let $\sL$ be a similarity type, let~$\cV$ be a variety of $\sL$-algebras, let $\kappa$ be an infinite cardinal, let $I_1,I_2,\dots,I_n$ be well-founded trees, and let~$J$ be a category, such that:
\begin{enumerate}
\item $\cV$ is locally $\le\kappa$.
\item $\card I_1 + \card I_2 + \dots + \card I_{n-1}  + \card\Mor J\le\kappa$.
\item $\card I_n \le\kappa^+$.
\item $\card \dnw i\le\kappa$ for each $i\in I_n$.
\item Every diagram of \jzs s \textup(resp., \jzus s\textup) $\cD\colon J\to\cS$, such that $\card\cD(j)\le\kappa^{+n}$, has a lifting in~$\cV$.
\end{enumerate}
Then every diagram of \jzs s \textup(resp., \jzus s\textup) $\cD\colon I_1\times I_2\times \dots\times I_n\times J\to\cS$, such that $\card\cD(i_1,i_2,\dots,i_n,j)\le\kappa$ for all $(i_1,i_2,\dots,i_n,j) \in I_1\times I_2\times\dots\times I_n\times\Ob J$, has a lifting in~$\cV$.
\end{corollary}

\begin{corollary}\label{C:AllsemilatticesToAllDiagramProductOfTrees}
Let $\cS$ be a class of \jzs s \textup(resp., \jzus s\textup) closed under finite products and directed unions \textup(resp., directed unions preserving 0, and 1\textup), let $\sL$ be a similarity type, let~$\cV$ be a variety of $\sL$-algebras. If every $S\in\cS$ has a lifting in~$\cV$, then every diagram of \jzs s \textup(resp., \jzus s\textup) of $\cS$, indexed by a finite product of well-founded trees, has a lifting in~$\cV$.
\end{corollary}

Using the result of Lampe in \cite{Lamp82}, that is, every \jzus\ is the compact congruence semilattice of a groupoid, we obtain a generalization of his result of simultaneous representation in \cite{Lamp06}, to all diagrams of \jzus s indexed by a finite poset.

\begin{corollary}\label{C:generalizedLamp}
Let $\cS$ be the category of all \jzus s with $(\vee,0,1)$-homomorphisms, let $I$ be a finite poset. Then every diagram $\cD\colon I\to\cS$ has a lifting in the variety of all groupoids.
\end{corollary}

\begin{proof}
We denote by~$\cV$ the variety of all groupoids. Remember that~$\cV$ has all small colimits (cf. Section~\ref{S:Basic}). For $I=2^n$, for a positive integer~$n$, the result follows from Corollary~\ref{C:AllsemilatticesToAllDiagramProductOfTrees}. Now let $I$ be an arbitrary finite poset and let $\cD\colon I\to\cS$ be a diagram of \jzus s. Put $S_X=\varinjlim(\cD\res X)=\varinjlim_{i\in X}\cD(i)$, for each $X\in\Pow(I)$. Let $s_{X,Y}\colon S_X\to S_Y$ be the canonical morphism, for all $X\subseteq Y\subseteq I$. Then 
\begin{align*}
\cD'\colon \Pow(I) &\to \cS\\
X &\mapsto S_X, & & \text{for all $X\in \Pow(I)$}\\
(X\subseteq Y)& \mapsto s_{X,Y}, & & \text{for all $X\subseteq Y \subseteq I$}
\end{align*}
is a functor. As $\Pow(I)\cong 2^I$, there exists a lifting $\cA'\colon \Pow(I)\to\cV$ of $\cD'$. Moreover, as $S_{I\dnw i}=\cD(i)$ and $s_{I\dnw i, I\dnw j}=\cD(i\le j)$ for all $i\le j$ in $I$, the functor
\begin{align*}
\cA\colon I &\to \cV\\
i &\mapsto \cA'(I\dnw i), & & \text{for all $i\in I$}\\
(i\le j)& \mapsto \cA'(I\dnw i\subseteq I\dnw j), & & \text{for all $i\le j \in I$}
\end{align*}
is a lifting of~$\cD$.
\end{proof}

In particular, consider the diagram denoted by~$\cD_{\bowtie}$ in~\cite{Bowtie}. This diagram is a diagram of finite Boolean semilattices and \jzue s; it is indexed by the bounded poset with atoms $a_i$ and coatoms~$b_i$, for $i<3$, and $a_i<b_j$ for all $i,j<3$. It is proved in~\cite{Bowtie} that this diagram does not have any congruence-lifting in any variety of algebras satisfying a nontrivial congruence lattice identity. It was not known at that time whether~$\cD_{\bowtie}$ was congruence-liftable by groupoids. So, by Corollary~\ref{C:generalizedLamp}, this is the case.

Define a \emph{quasi-partition} of a set $X$ as a family $(Y_k)_{k\in K}$ of subsets of $X$ such that $X=\bigcup_{k\in K}Y_k$ and $Y_k\cap Y_l=\emptyset$ for all $k\not =l$ in $K$ (we do not require the~$Y_k$s to be nonempty).

The following result is a compactness-type property for liftings of diagrams.

\begin{theorem}\label{T:Compactness}
Let $\cS$ be the class of all distributive \jzs s, let~$\cV$ be a finitely generated congruence-distributive variety of algebras, let~$J$ be a small category, such that there are at most finitely many arrows between any two objects, let $\cD\colon J\to\cS$ be a functor such that $\cD(j)$ is finite for all $j\in J$. If every finite subdiagram of~$\cD$ has a lifting in~$\cV$, then~$\cD$ has a lifting in~$\cV$.
\end{theorem}

\begin{proof}
Let $(K_j)_{j\in J}$ be a family of finite subsets of~$\cV$ such that if $\Conc A$ is isomorphic to $\cD(j)$, for some $A\in \cV$ and $j\in\Ob J$, then~$A$ is isomorphic to an element of~$K_j$.

For every finite subset $I$ of $\Ob J$, we denote by $\overline{I}$ the full subcategory of $J$ with class of objects $I$. Let $\cA_I\colon \overline{I}\to \cV$ be a functor and let $\xi_I=(\xi_I^i)_{i\in \Ob I}\colon\Conc\circ\cA\to\cD\res I$ be a natural isomorphism. We can assume that $\cA_I(i)\in K_i$ for all $i\in I$.

Put $Q_S=\setm{P\in\Powf{\Ob J}}{S\subseteq P}$, and denote by $\fF$ the filter on $\Powf{\Ob J}$ generated by $\setm{Q_S}{S\in\Powf{\Ob J}}$. As $Q_{S_1}\cap Q_{S_2}=Q_{S_1\cup S_2}$ for all $S_1,S_2\in\Powf{\Ob J}$, the filter~$\fF$ is proper, so there exists an ultrafilter $\fU$ such that $\fF\subseteq \fU$.

Let $j\in\Ob J$. The family $(\setm{P\in Q_{\set{j}}}{\cA_P(j)=A})_{A\in K_j}$ is a finite quasi-partition of~$Q_{\set{j}}$, so there exists a unique $A_j\in K_j$ such that $R_j=\setm{P\in Q_{\set{j}}}{\cA_P(j)=A_j}$ belongs to $\fU$.

Let $f\colon i\to j$ be an arrow of $J$. The family $(\setm{P\in R_i\cap R_j}{\cA_P(f)=s})_{s\colon A_i\to A_j}$ is a finite quasi-partition of $R_i\cap R_j\in\fU$, so there exists a unique $s_f\colon A_i\to A_j$ such that $S_f=\setm{P\in R_i\cap R_j}{\cA_P(f)=s_f}$ belongs to $\fU$.

Let $i\in \Ob J$, let $P\in S_{\id_i}$, so $\cA_P(i)=A_i$ and $\id_{A_i}=\id_{\cA_P(i)}=\cA_P(\id_i)=s_{\id_i}$. Let $f\colon i\to j$ and $g\colon j\to k$ be two arrows of $J$, let $P\in S_f\cap S_g\cap S_{g\circ f}$. So $\cA_P(i)=A_i$, $\cA_P(j)=A_j$, and $\cA_P(k)=A_k$. Moreover:
\[s_g\circ s_f=\cA_P(g)\circ\cA_P(f)=\cA_P(f\circ g)=s_{g\circ f}.\]
Thus we obtain a functor:
\begin{align*}
\cA\colon J & \to \cV\\
i & \mapsto A_i & &\text{for all $i\in\Ob J$}\\
f & \mapsto s_f & &\text{for all $f\in\Mor J$}
\end{align*}

For each $j\in\Ob J$, the family $(\setm{P\in R_j}{\xi_P^j=\phi})_{\phi\colon \Conc\cA(j)\to\cD(j)}$ is a finite quasi-partition of $R_j$, so there exists a unique $\phi_j\colon \Conc\cA(j)\to\cD(j)$ such that the set $T_j=\setm{P\in R_j}{\xi_P^j=\phi_j}$ belongs to $\fU$.

Let $f\colon i\to j$ be an arrow of $J$, let $P\in S_f\cap T_i\cap T_j$. So the following equalities hold:
\[\phi_j\circ\Conc\cA(f) = \xi_P^j\circ\Conc\cA_P(f)=\cD(f)\circ\xi_P^i=\cD(f)\circ\phi_i,\]
and so $(\phi_j)_{j\in\Ob J}\colon \Conc\circ\cA\to \cD$ is a natural isomorphism. Thus~$\cD$ has a lifting in~$\cV$.
\end{proof}

The following corollary gives us, in particular, a characterization of all pairs of finitely generated congruence-distributive varieties with uncountable critical point.

\begin{corollary}\label{C:caract-crit-point-aleph0-inFGCDV}
Let~$\cV_1$ be a locally finite variety, let~$\cV_2$ be a finitely generated congruence-distributive variety. Then the following statements are equivalent:
\begin{enumerate}
\item $\crit{\cV_1}{\cV_2}>\aleph_0$.
\item Every diagram of finite \jzs s indexed by a tree which has a lifting in~$\cV_1$ has a lifting in~$\cV_2$.
\item Every diagram of finite \jzs s indexed by a finite chain which has a lifting in~$\cV_1$ has a lifting in~$\cV_2$.
\end{enumerate}
\end{corollary}

\begin{proof}
If $(1)$ holds, then by Corollary~\ref{C:lifting-inFGCDvariety} every diagram of finite \jzs s indexed by a finite tree which has a lifting in~$\cV_1$ has a lifting in~$\cV_2$. Thus, by Theorem~\ref{T:Compactness}, the statement~$(2)$ holds.

Now assume that $(3)$ holds. By Theorem~\ref{T:Compactness}, every diagram of finite \jzs s indexed by~$\omega$ which has a lifting in~$\cV_1$ has a lifting in~$\cV_2$. Let $D$ be a countable distributive \jzs. Let $A\in\cV_1$ such that $\Conc A\cong D$. Using Lemma~\ref{L:LS} we can assume that~$A$ is countable. So we can write $A=\bigcup_{n\in\omega} A_n$, where each~$A_n$ is a finite subalgebra of~$A$, and $A_m\subseteq A_n$ for all $m\le n$ in~$\omega$. Denote by $f_{m,n}\colon A_m\to A_n$ the inclusion map, for all $m\le n$ in~$\omega$. Put $\cA=((A_n)_{n\in\omega},(f_{m,n})_{m\le n\in\omega})$. So we get a diagram $\cB=((B_n)_{n\in\omega},(g_{m,n})_{m\le n\in\omega})$ in~$\cV_2$ together with a natural isomorphism $\xi\colon \Conc\circ\cA\to\Conc\circ\cB$. Hence, as the $\Conc$ functor preserves direct limits,
\[\Conc A\cong\Conc(\varinjlim \cA)\cong\varinjlim(\Conc\circ \cA)\cong\varinjlim(\Conc\circ \cB)\cong \Conc(\varinjlim\cB).\tag*{\qed}\]
\renewcommand{\qed}{}
\end{proof}

\begin{corollary}\label{C:caract-crit-point-infty-inFGCDV}
Let~$\cV_1$ be a locally finite variety, let~$\cV_2$ be a finitely generated congruence-distributive variety. Then the following statements are equivalent:
\begin{enumerate}
\item $\crit{\cV_1}{\cV_2}\ge\aleph_\omega$.
\item Every diagram of finite \jzs s indexed by $\set{0,1}^n$, for a natural number $n$, which has a lifting in~$\cV_1$ has a lifting in~$\cV_2$.
\item Every diagram of finite \jzs s indexed by a finite \jzs\ which has a lifting in~$\cV_1$ has a lifting in~$\cV_2$.
\item Every diagram of finite \jzs s indexed by a \jzs\ which has a lifting in~$\cV_1$ has a lifting in~$\cV_2$.
\item Every diagram of \jzs s indexed by a \jzs\ which has a lifting in~$\cV_1$ has a lifting in~$\cV_2$.
\item $\crit{\cV_1}{\cV_2}=\infty$, that is, $\Conc\cV_1\subseteq\Conc\cV_2$.
\end{enumerate}
\end{corollary}

\begin{proof}
By Corollary~\ref{C:lifting-inFGCDvariety-with-alephn}, the statement $(1)\Longrightarrow(2)$ holds. By Theorem~\ref{T:Compactness}, the statement $(3)\Longrightarrow(4)$ holds. The statements $(5)\Longrightarrow(6)$ and $(6)\Longrightarrow(1)$ are obvious.
Denote by $\cS$ the class of all distributive \jzs s. Now assume that $(2)$ holds. Let~$L$ be a finite \jzs, let~$\cD$ be a diagram of finite \jzs s indexed by~$L$, let $\cA\colon L\to\cV_1$ be a lifting of~$\cD$. Put:
\begin{align*}
\cD'\colon\Pow(L)&\to \cS\\
X&\mapsto \cD(\bigvee X)\\
X\subseteq Y &\mapsto \cD(\bigvee X\le \bigvee Y).
\end{align*}
This is a functor. Moreover, the functor $\cA'\colon\Pow(L)\to \cV_1$ defined by
\begin{align*}
X&\mapsto \cA(\bigvee X)\\
X\subseteq Y &\mapsto \cA(\bigvee X\le \bigvee Y)
\end{align*}
is a lifting of $\cD'$. So, by $(2)$, there exists a lifting $\cB'\colon\Pow(L)\to \cV_2$ of $\cD'$. Moreover:
\begin{align*}
\cB\colon L & \to\cV_2\\
x & \mapsto\cB'(L\dnw x) & &\text{for all $x\in L$}\\
(x\le y) & \mapsto \cB'(L\dnw x\subseteq L\dnw y) & &\text{for all $x\le y \in L$}
\end{align*}
is a lifting of~$\cD$. This completes the proof of $(3)$.

Now assume $(4)$. Let $L$ be a \jzs, let $\cD\colon L\to \cS$ be a functor, let $\cA\colon L\to\cV_1$ be a lifting of~$\cD$. Fix $a\in\cA(0)$.
Let:
\begin{align*}
G= &\{(Q_x)_{x\in L}\mid \text{$Q_x$ is a finite subalgebra of $\cA(x)$, for all $x\in L$,}\\
&\quad\quad\quad\quad\quad\text{$\cA(x\le y)(Q_x)\subseteq Q_y$, for all $x\le y\in L$, and $a\in Q_0$}
\}
\end{align*}
partially ordered by $(Q_x)_{x\in L}\le (Q'_x)_{x\in L}$ if $Q_x\subseteq Q'_x$ for all $x\in L$. The subalgebra $\langle\cA(0\le x)(a)\rangle_{\cA(x)}$ of $\cA(x)$ generated by $\cA(0\le x)(a)$ is finitely generated, thus finite (because~$\cV_1$ is locally finite). Thus $G$ is a \jzs\ with smallest element $\big(\langle\cA(0\le x)(a)\rangle_{\cA(x)}\big)_{x\in L}$.

Let:
\begin{align*}
\cA'\colon G\times L &\to \cV_1\\
(Q,x) & \mapsto Q_x\\
((Q,x)\le (Q',x')) &\mapsto (\cA(x\le x')\res Q_x\colon Q_x\to Q_y)
\end{align*}
Consider $\widetilde\cA'\colon L\to\cV_1^G$ as defined in Section \ref{S:Basic}. Then:
\[ \varinjlim\widetilde\cA'(x) = \bigcup_{Q\in G}\cA'(Q,x) = \cA(x),\quad\text{for all $x\in L$}\]
\[ \varinjlim\widetilde\cA'(x\le y) = \bigcup_{Q\in G}\cA'((Q,x)\le (Q,y)) = \cA(x\le y),\quad\text{for all $x\le y\in L$}\]
As $\Conc\cA'$ has a lifting in~$\cV_1$, it has also a lifting $\cB'\colon G\times L$ in~$\cV_2$. Let
\[\cB=\varinjlim\circ\widetilde\cB'\colon L\to\cV_2.\]
As $\Conc$ preserves direct limits, the following natural isomorphisms hold:
\begin{align*}
\cD &\cong \Conc\circ\cA\\
&\cong \Conc\circ\varinjlim\circ \widetilde\cA'\\
&\cong \varinjlim\circ\Conc\circ \widetilde\cA'\\
&\cong \varinjlim\circ\Conc\circ \widetilde\cB'\\
&\cong \Conc\circ\varinjlim\circ \widetilde\cB'\\
&\cong \Conc\circ\cB.\tag*{\qed}
\end{align*}
\renewcommand{\qed}{}
\end{proof}

\begin{corollary}\label{C:critpointifFGCDVisalephn}
Let~$\cV_1$ be a locally finite variety, let~$\cV_2$ be a finitely generated congruence-distributive variety. Then exactly one of the following statements holds:
\begin{enumerate}
\item $\crit{\cV_1}{\cV_2}$ is finite.
\item $\crit{\cV_1}{\cV_2}=\aleph_n$, for some natural number $n$.
\item $\crit{\cV_1}{\cV_2}=\infty$, that is, $\Conc\cV_1\subseteq\Conc\cV_2$.
\end{enumerate}
\end{corollary}

\section{A pair of varieties with critical point $\aleph_1$}\label{S:critpointaleph1}

\begin{lemma}\label{L:ConABoolean}
Let~$A$ be a finite algebra with $\Con A$ distributive, let $\alpha\in\Con A$, and put $Q=\setm{\theta\in\M(\Con A)}{\alpha\not\le\theta}$. If all $A/\theta$, for $\theta\in Q$, are simple, then the canonical map $\Con A\to\Con(A/\alpha)\times\prod_{\theta\in Q}\Con(A/\theta)$ is an isomorphism.
\end{lemma}

\begin{proof}
As $\Con(A/\xi)\cong\upw\xi$, for all $\xi\in\Con A$, it suffices to prove that the map $j\colon\Con A\to(\upw\alpha)\times\prod_{\theta\in Q}(\upw\theta)$, $\xi\mapsto(\xi\vee\alpha,(\xi\vee\theta)_{\theta\in Q})$ is an isomorphism. If $\alpha\wedge\bigwedge Q\not=0$, then there exists $\theta\in\M(\Con A)$ such that $\alpha\wedge\bigwedge Q\not\le\theta$, thus $\alpha\not\le\theta$ (thus $\theta\in Q$) and $\bigwedge Q\not\le\theta$, a contradiction; whence $\alpha\wedge\bigwedge Q=0$. By using the distributivity of $\Con A$, it follows that $j$ is one-to-one.

We now prove that $j$ is surjective. Let $\beta\in\upw\alpha$, let $\gamma_\theta\in\upw\theta$ for all $\theta\in Q$. Put $\xi=\beta\wedge\bigwedge_{\theta\in Q}\gamma_\theta$. We have $\alpha\vee\beta=\beta$ and $\alpha\vee\theta=A\times A$ for all $\theta\in Q$, so:
\[
\xi\vee\alpha=(\beta\wedge\bigwedge_{\theta\in Q}\gamma_\theta)\vee\alpha=(\beta\vee\alpha)\wedge\bigwedge_{\theta\in Q}(\gamma_\theta\vee\alpha)=\beta\wedge\bigwedge_{\theta\in Q}(A\times A)=\beta
\]
With a similar argument we obtain $\xi\vee\theta=\gamma_\theta$ for all $\theta\in Q$, thus $j$ is surjective. Therefore, $j$ is an isomorphism.
\end{proof}

We say that a class~$\cV$ of algebras of the same similarity type is \emph{finitely semisimple}, if every finite subdirectly irreducible member of~$\cV$ is simple. An important example of a finitely semisimple variety is the variety of all modular lattices.

\begin{lemma}\label{L:CS_critpoint_ge_aleph_1}
Let~$\cV_1$ and~$\cV_2$ be congruence-distributive varieties of algebras of the same similarity type, with~$\cV_1$ finitely semisimple. We further assume that for every finite non-simple algebra $A\in\cV_1$, if~$A$ embeds into a simple algebra of~$\cV_1$, then~$A$ embeds into a simple algebra of~$\cV_2$.

Let $f\colon A\to A'$ be a morphism between finite algebras of~$\cV_1$. We denote by $\alpha$ \textup(resp., $\alpha'$\textup) the smallest congruence of~$A$ \textup(resp., $A'$\textup) such that $A/\alpha\in\cV_2$ \textup(resp., $A'/\alpha'\in\cV_2$\textup), with canonical projection $\pi_\alpha\colon A\tosurj A/\alpha$ \textup(resp., $\pi_{\alpha'}'\colon A'\tosurj A'/\alpha'$\textup). Let $B\in\cV_2$, let $p\colon B\tosurj A/\alpha$ be a surjective morphism, and let $\xi\colon\Conc A\to\Conc B$ be an isomorphism such that $(\Conc p)\circ\xi=\Conc\pi_\alpha$. Then there are $B'\in\cV_2$, a morphism $g\colon B\to B'$, a surjective morphism $p'\colon B'\tosurj A'/\alpha'$, and an isomorphism $\xi'\colon\Conc A'\to\Conc B'$, such that the following diagram is commutative:
\[
\xymatrix{
 & \Conc A \ar[rr]^{\Conc f} \ar[dl]_{\Conc\pi_\alpha}\ar[dd]^{\xi}_{\cong} & & \Conc A' \ar[dd]^{\xi'}_{\cong} \ar[dr]^{\Conc\pi'_{\alpha'}}&\\
\Conc(A/\alpha) & & & & \Conc(A'/\alpha')\\
 & \Conc B \ar@{->>}[ul]^{\Conc p} \ar[rr]^{\Conc g} & & \Conc B' \ar@{->>}[ur]_{\Conc p'}
}
\]

If there is at least one simple algebra in~$\cV_2$, then $\Conc\circ\cA$ has a lifting in~$\cV_2$, for every diagram $\cA\colon\omega\to\cV_1$ of finite algebras, 

Moreover, if~$\cV_1$ is locally finite, then $\crit{\cV_1}{\cV_2}\ge\aleph_1$.
\end{lemma}

\begin{proof}
We denote by $\pi_\theta\colon A\tosurj A/\theta$ (resp., $\pi'_\theta\colon A'\tosurj A'/\theta$) the canonical projection for each $\theta\in\Con A$ (resp., $\theta\in\Con A'$).
The algebra $A/f^{-1}(\alpha')$ is isomorphic to a subalgebra of $A'/\alpha'\in\cV_2$, thus $A/f^{-1}(\alpha')\in\cV_2$, so $f^{-1}(\alpha')\supseteq\alpha$, and so $\Conc(f)(\alpha)\subseteq\alpha'$, thus the morphism $g_\alpha\colon A/\alpha\to A'/\alpha'$, $x/\alpha\mapsto f(x)/\alpha'$ is well-defined, and the following diagram is commutative:
\[
\begin{CD}
A @>f>> A'\\
@V{\pi_\alpha}VV @V{\pi'_{\alpha'}}VV\\
A/\alpha @>{g_\alpha}>> A'/\alpha'
\end{CD}
\]
Put $h_\alpha = g_\alpha\circ p$.

Put $Q=\setm{\theta\in\M(\Con A')}{A'/\theta\not\in\cV_2}$. For each $\theta\in Q$, the algebra $A/f^{-1}(\theta)$ is isomorphic to a subalgebra of $A'/\theta$ which is a simple algebra of~$\cV_1$. If $A/f^{-1}(\theta)$ is not simple, then $A/f^{-1}(\theta)\in\cV_2$,  and $A/f^{-1}(\theta)$ is a subalgebra of a simple algebra of~$\cV_2$. So one of the following statements holds:
\begin{enumerate}
\item The algebra $A/f^{-1}(\theta)$ is a subalgebra of a simple algebra in~$\cV_2$.
\item The algebra $A/f^{-1}(\theta)$ is simple and is not in~$\cV_2$.
\end{enumerate}

If $A/f^{-1}(\theta)\not\in\cV_2$, let $B_\theta=B/\xi(f^{-1}(\theta))$, which is a simple algebra, and let $h_\theta\colon B\tosurj B_\theta$ be the canonical projection. If $A/f^{-1}(\theta)\in\cV_2$, then there are a simple algebra $B_\theta\in\cV_2$ and an embedding $g_\theta=A/f^{-1}(\theta)\toinj B_\theta$. Moreover, as $A/f^{-1}(\theta)\in\cV_2$, the containment $f^{-1}(\theta)\supseteq \alpha$ holds. Denote by $p_{\theta}\colon A/\alpha\tosurj A/f^{-1}(\theta)$ the canonical projection. Put $h_\theta=g_\theta\circ p_\theta\circ p$.

Let $\phi_\theta \colon \Conc(A'/\theta)\to \Conc B_\theta$ be the only possible isomorphism, put $\xi'_\theta =\phi_\theta\circ\Conc\pi'_\theta$, for all $\theta\in Q$. Let $\xi'_{\alpha'}=\Conc\pi'_{\alpha'}$.

The algebra $B'=A'/\alpha'\times\prod_{\theta\in Q} B_\theta$ belongs to~$\cV_2$. Define
\begin{align*}
g\colon B&\to B'\\
x&\mapsto (h_\alpha(x),(h_\theta(x))_{\theta\in Q}).
\end{align*}
Observe that as $\Con B'\cong \Con (A'/\alpha')\times\prod_{\theta\in Q} \Con (B_\theta)$ is finite, every congruence of $B'$ is compact, so $\Conc B'=\Con B'$, thus we can define a map
\begin{align*}
\xi'\colon \Conc A' &\to \Conc B'\\
\bx&\mapsto \xi'_{\alpha'}(\bx) \times \prod_{\theta\in Q}\xi'_\theta(\bx).
\end{align*}
By Lemma~\ref{L:ConABoolean}, the canonical map $\psi\colon\Conc A'\to \Conc(A'/\alpha')\times\prod_{\theta\in Q}\Conc(A'/\theta)$ is an isomorphism, the map $\id_{\Conc(A'/\alpha')}\times\prod_{\theta\in Q}\phi_\theta$ is also an isomorphism, so the map $\xi'=(\id_{\Conc (A'/\alpha')}\times\prod_{\theta\in Q}\phi_\theta)\circ\psi$ is an isomorphism.

Denote by $p'\colon B'\tosurj A'/\alpha'$ the canonical projection and by $p'_\theta\colon B'\to B_\theta$ the canonical projection, for all $\theta\in Q$.

The equality $(\Conc p')\circ\xi'=\xi'_{\alpha'}$ is obvious. Moreover $p'\circ g=g_\alpha\circ p$, so the following equalities hold:
\begin{equation}\label{E:Eq1LemCritpointGEAleph0}
(\Conc p')\circ (\Conc g)\circ\xi=(\Conc g_\alpha)\circ(\Conc p)\circ\xi=(\Conc g_\alpha)\circ\Conc\pi_\alpha.
\end{equation}
As $g_\alpha\circ \pi_\alpha= \pi'_{\alpha'}\circ f$ we obtain
\[
(\Conc p')\circ (\Conc g)\circ\xi=(\Conc\pi'_{\alpha'})\circ \Conc f = (\Conc p')\circ \xi'\circ\Conc f.
\]

Let $\theta\in Q$, then the following equalities hold:
\begin{equation*}
(\Conc p'_\theta)\circ\xi'\circ(\Conc f)=\xi'_\theta\circ(\Conc f)=\phi_\theta\circ(\Conc\pi'_\theta)\circ(\Conc f).
\end{equation*}

Assume that $A/f^{-1}(\theta)\not\in\cV_2$. Let $\beta\in\Conc A$, then the following equivalences hold:
\begin{align*}
((\Conc p'_\theta)\circ\xi'\circ(\Conc f))(\beta)=0 & \Longleftrightarrow \Conc(\pi'_\theta\circ f)(\beta)=0 \\
&\Longleftrightarrow \beta\subseteq f^{-1}(\theta) \\
&\Longleftrightarrow \xi(\beta)\subseteq \xi(f^{-1}(\theta)) \\
&\Longleftrightarrow (\Conc h_\theta)(\xi(\beta)) = 0 \\
&\Longleftrightarrow ((\Conc p'_\theta)\circ(\Conc g)\circ\xi)(\beta) = 0.
\end{align*}
Therefore, as $B_\theta$ is simple, we obtain
\begin{align}
&(\Conc p'_\theta)\circ\xi'\circ(\Conc f)=(\Conc p'_\theta)\circ(\Conc g)\circ\xi,\notag\\
\label{E:Eq2LemCritpointGEAleph0}
 & \text{for all $\theta\in Q$ such that $A/f^{-1}(\theta)\not\in\cV_2$.}
\end{align}

Assume that $A/f^{-1}(\theta)\in\cV_2$. The following equalities hold:
\begin{align*}
(\Conc p'_\theta)\circ(\Conc g)\circ\xi &= (\Conc h_\theta)\circ\xi  \\
&=(\Conc g_\theta)\circ(\Conc p_\theta)\circ(\Conc p)\circ\xi \\
&=(\Conc g_\theta)\circ(\Conc p_\theta)\circ(\Conc \pi_\alpha) \\
&=(\Conc g_\theta)\circ(\Conc \pi_{f^{-1}(\theta)}).
\end{align*}
Let $\beta\in\Conc A$, the following equivalences hold:
\begin{align*}
(\Conc &p'_\theta)\circ(\Conc g)\circ\xi(\beta)=0\\
& \Longleftrightarrow (\Conc g_\theta)\circ(\Conc \pi_{f^{-1}(\theta)})(\beta)=0\\
&\Longleftrightarrow (\Conc \pi_{f^{-1}(\theta)})(\beta)=0 & & \text{as $g_\theta$ is one-to-one}\\
&\Longleftrightarrow \beta\subseteq f^{-1}(\theta)\\
& \Longleftrightarrow \Conc(\pi'_\theta\circ f)(\beta)=0\\
& \Longleftrightarrow (\Conc p'_\theta)\circ\xi'\circ(\Conc f)(\beta)=0.
\end{align*}
Therefore, as $B_\theta$ is simple,
\begin{align}
&(\Conc p'_\theta)\circ\xi'\circ(\Conc f)=(\Conc p'_\theta)\circ(\Conc g)\circ\xi,\notag\\
\label{E:Eq3LemCritpointGEAleph0}
&\text{for all $\theta\in Q$ such that $A/f^{-1}(\theta)\in\cV_2$.}
\end{align}
As $\Conc B'\toinj\Conc(A/\alpha')\times\prod_{\theta\in Q}\Conc B_\theta$, by \eqref{E:Eq1LemCritpointGEAleph0}, \eqref{E:Eq2LemCritpointGEAleph0}, and \eqref{E:Eq3LemCritpointGEAleph0} the following diagram is commutative:
\[
\begin{CD}
\Conc A @>\Conc f>> \Conc A'\\
@V{\xi}VV   @V{\xi'}VV \\
\Conc B @>\Conc g>> \Conc B'
\end{CD}
\]

Let $S$ be a simple algebra in~$\cV_2$, let $\cA\colon\omega\to\cV_1$ be a diagram of finite algebras, let $\alpha_n$ be the smallest congruence of $\cA(n)$ such that $\cA(n)/\alpha_n\in\cV_2$, let $\pi^n_\theta\colon \cA(n)\tosurj \cA(n)/\theta$ be the canonical projection, for all $\theta\in\Con\cA(n)$. Let $Q_n=\setm{\theta\in\M(\Conc \cA(n))}{\cA(n)/\theta\not\in \cV_2}$, for all $n\in\omega$. Let $\phi_\theta \colon\Conc (\cA(0)/\theta)\to S$ be the only possible isomorphism. Let $\xi_{\alpha_0}=\Conc\pi^0_{\alpha_0}$, let $\xi_\theta=\phi_\theta\circ\Conc\pi^0_\theta$, for all $\theta\in Q_0$. Put $B_0=(\cA(0)/\alpha_0)\times S^{Q_0}$, let $p_0\colon B_0\tosurj \cA(0)/\alpha_0$ be the canonical projection. By Lemma~\ref{L:ConABoolean}, the morphism
\begin{align*}
\xi_0\colon \Conc \cA(0) &\to \Conc B_0\\
\bx &\mapsto \xi_{\alpha_0}(\bx)\times\prod_{\theta\in Q_0}\xi_\theta(\bx)
\end{align*}
is an isomorphism. Moreover $(\Conc p_0)\circ\xi_0 =\xi_{\alpha_0}=\Conc\pi_{\alpha_0}^0$. Thus, applying by induction the first part of the lemma, we construct a family $(B_n)_{n\in\omega}$ of algebras of~$\cV_2$, a family $(g_n\colon B_n\to B_{n+1})_{n\in\omega}$ of homomorphisms, and a family $(\xi_n\colon\Conc\cA(n)\to B_n)_{n\in\omega}$ of isomorphisms such that the following diagram is commutative:
\begin{equation}
\xymatrix{
\Conc\cA(n) \ar[rrr]^-{\Conc\cA(n\le n+1)} \ar[d]_-{\xi_n} & & &\Conc\cA(n+1) \ar[d]^-{\xi_{n+1}}\\
\Conc B_n \ar[rrr]_-{\Conc g_n} & & & \Conc B_{n+1}
}
\end{equation}
Then the functor
\begin{align*}
\cB\colon\omega &\to \cV_2\\
n &\mapsto B_n\\
(n\le m)&\mapsto g_{m-1}\circ g_{m-2}\circ\dots\circ g_n
\end{align*}
is a lifting of $\Conc\circ\cA$ in~$\cV_2$.

Now assume that~$\cV_1$ is locally finite. Let $A\in\cV_1$ such that $\Conc A$ is countable. Taking a sublattice, we can assume that~$A$ is countable (cf. Lemma~\ref{L:LS}) and so it is the direct limit of a diagram $\cA\colon\omega\to\cV_1$ of finite algebras. So $\Conc\circ \cA$ has a lifting in~$\cV_2$, thus, as $\Conc$ preserves direct limits, $\Conc A$ has a lifting in~$\cV_2$. So $\crit{\cV_1}{\cV_2}>\aleph_0$.
\end{proof}

\begin{remark}\label{R:concsurj}
Let $f\colon K\toinj L$ be a one-to-one morphism of finite modular lattices, such that $K$ and $L$ have the same length; then $\Conc f$ is surjective.
\end{remark}

\begin{corollary}
Let~$\cV_1$ be the variety generated by $T_1$, let~$\cV_2$ be the variety generated by $T_2$, $T_3$, and $T_4$, where $T_1$, $T_2$, $T_3$, and $T_4$ are the lattices in Figure \textup{\ref{F:treillis_T_1_T_2_T_3_T_4}}. Then $\crit{\cV_1}{\cV_2}=\aleph_1$. This result extends to the corresponding varieties of bounded lattices \pup{resp., lattices with zero}.
\end{corollary}

Observe that the varieties $\cV_1$ and~$\cV_2$ are self-dual.

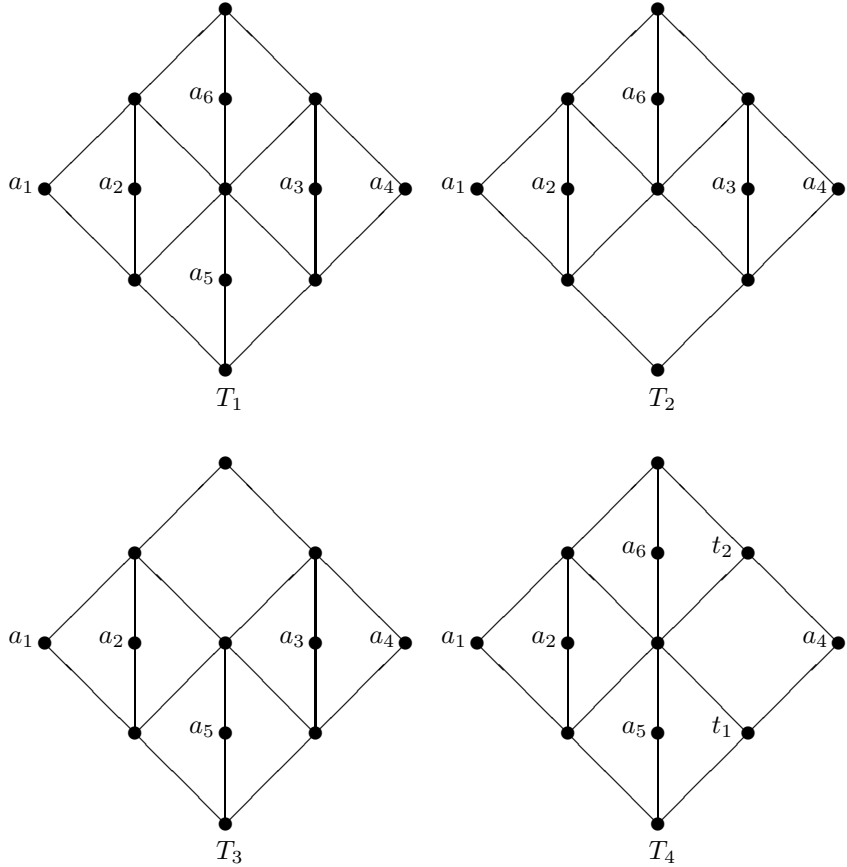
\begin{figure}[here,bottom,top]
\caption{The lattices $T_1$, $T_2$, $T_3$, and $T_4$.}\label{F:treillis_T_1_T_2_T_3_T_4}

\setlength{\unitlength}{0.6mm}
\begin{picture}(90,100)(-10,-10)
\put(40,0){\line(-1,1){40}}
\put(40,0){\line(1,1){40}}
\put(40,0){\line(0,1){80}}

\put(0,40){\line(1,1){40}}
\put(80,40){\line(-1,1){40}}
\put(20,20){\line(1,1){40}}
\put(60,20){\line(-1,1){40}}

\put(20,20){\line(0,1){40}}
\put(60,20){\line(0,1){40}}

\put(40,0){\circle*{3}}

\put(20,20){\circle*{3}}
\put(40,20){\circle*{3}}
\put(60,20){\circle*{3}}

\put(0,40){\circle*{3}}
\put(20,40){\circle*{3}}
\put(40,40){\circle*{3}}
\put(60,40){\circle*{3}}
\put(80,40){\circle*{3}}

\put(20,60){\circle*{3}}
\put(40,60){\circle*{3}}
\put(60,60){\circle*{3}}
\put(40,80){\circle*{3}}

\put(-8,40){$a_1$}
\put(12,40){$a_2$}

\put(52,40){$a_3$}
\put(72,40){$a_4$}

\put(32,20){$a_5$}
\put(32,60){$a_6$}

\put(38,-8){$T_1$}
\end{picture}\quad
\begin{picture}(90,100)(-10,-10)
\put(40,0){\line(-1,1){40}}
\put(40,0){\line(1,1){40}}
\put(40,40){\line(0,1){40}}

\put(0,40){\line(1,1){40}}
\put(80,40){\line(-1,1){40}}
\put(20,20){\line(1,1){40}}
\put(60,20){\line(-1,1){40}}

\put(20,20){\line(0,1){40}}
\put(60,20){\line(0,1){40}}

\put(40,0){\circle*{3}}

\put(20,20){\circle*{3}}
\put(60,20){\circle*{3}}

\put(0,40){\circle*{3}}
\put(20,40){\circle*{3}}
\put(40,40){\circle*{3}}
\put(60,40){\circle*{3}}
\put(80,40){\circle*{3}}

\put(20,60){\circle*{3}}
\put(40,60){\circle*{3}}
\put(60,60){\circle*{3}}
\put(40,80){\circle*{3}}

\put(-8,40){$a_1$}
\put(12,40){$a_2$}

\put(52,40){$a_3$}
\put(72,40){$a_4$}

\put(32,60){$a_6$}

\put(38,-8){$T_2$}
\end{picture}
\begin{picture}(90,100)(-10,-10)
\put(40,0){\line(-1,1){40}}
\put(40,0){\line(1,1){40}}
\put(40,0){\line(0,1){40}}

\put(0,40){\line(1,1){40}}
\put(80,40){\line(-1,1){40}}
\put(20,20){\line(1,1){40}}
\put(60,20){\line(-1,1){40}}

\put(20,20){\line(0,1){40}}
\put(60,20){\line(0,1){40}}

\put(40,0){\circle*{3}}

\put(20,20){\circle*{3}}
\put(40,20){\circle*{3}}
\put(60,20){\circle*{3}}

\put(0,40){\circle*{3}}
\put(20,40){\circle*{3}}
\put(40,40){\circle*{3}}
\put(60,40){\circle*{3}}
\put(80,40){\circle*{3}}

\put(20,60){\circle*{3}}
\put(60,60){\circle*{3}}
\put(40,80){\circle*{3}}

\put(-8,40){$a_1$}
\put(12,40){$a_2$}

\put(52,40){$a_3$}
\put(72,40){$a_4$}

\put(32,20){$a_5$}

\put(38,-8){$T_3$}
\end{picture}\quad
\begin{picture}(90,100)(-10,-10)
\put(40,0){\line(-1,1){40}}
\put(40,0){\line(1,1){40}}
\put(40,0){\line(0,1){80}}

\put(0,40){\line(1,1){40}}
\put(80,40){\line(-1,1){40}}
\put(20,20){\line(1,1){40}}
\put(60,20){\line(-1,1){40}}

\put(20,20){\line(0,1){40}}

\put(40,0){\circle*{3}}

\put(20,20){\circle*{3}}
\put(40,20){\circle*{3}}
\put(60,20){\circle*{3}}

\put(0,40){\circle*{3}}
\put(20,40){\circle*{3}}
\put(40,40){\circle*{3}}
\put(80,40){\circle*{3}}

\put(20,60){\circle*{3}}
\put(40,60){\circle*{3}}
\put(60,60){\circle*{3}}
\put(40,80){\circle*{3}}

\put(52,20){$t_1$}
\put(52,60){$t_2$}

\put(-8,40){$a_1$}
\put(12,40){$a_2$}

\put(72,40){$a_4$}

\put(32,20){$a_5$}
\put(32,60){$a_6$}

\put(38,-8){$T_4$}
\end{picture}
\end{figure}

\begin{proof}
The lattice $T_1$ is generated by $a_1,a_2,a_3,a_4,a_5$, and $a_6$ which are all doubly irreducible. So the maximal sublattices of $T_1$ are $T_1-\set{a_k}$, for all $1\le k\le 6$. As all these lattices are isomorphic to either $T_2$, $T_3$, or $T_4$, the assumptions of Lemma~\ref{L:CS_critpoint_ge_aleph_1} are satisfied, thus $\crit{\cV_1}{\cV_2}\ge\aleph_1$.

Put $D_0=2^4$, $D_1=D_2=2^2$, $D_3=2$. Put:
\begin{align*}
\phi_1\colon D_0&\rightarrow D_1\\
(\alpha,\beta,\gamma,\delta)&\mapsto(\alpha\vee\beta,\gamma\vee\delta)\\
\\
\phi_2\colon D_0&\rightarrow D_2\\
(\alpha,\beta,\gamma,\delta)&\mapsto(\alpha\vee\delta,\beta\vee\gamma)\\
\\
\psi\colon 2^2&\rightarrow D_3\\
(\alpha,\beta)&\mapsto\alpha\vee\beta
\end{align*}
Let $\overrightarrow{D}$ be the following commutative diagram:
$$ 
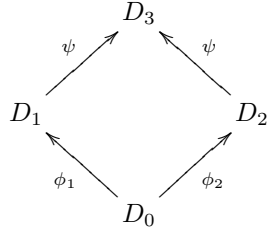
\xymatrix{
& D_3&\\ 
 D_1\ar[ur]^{\psi} & & D_2\ar[ul]_{\psi} \\
 & D_0 \ar[ul]^{\phi_1} \ar[ur]_{\phi_2}
}$$
Put $S_1=T_1-\{a_2,a_3\}$, and $S_2=T_1-\{a_5,a_6\}$. Then $S_1$ and $S_2$ are sublattices of $T_1$; put $S_0=S_1\cap S_2$. Let $i_1\colon S_0\to S_1$, $i_2\colon S_0\to S_2$, $j_1\colon S_1\to T_1$, $j_2\colon S_2\to T_1$ be the respective inclusion mappings. Then the following diagram is a lifting of~$\overrightarrow{D}$ in~$\cV_1$.
$$ 
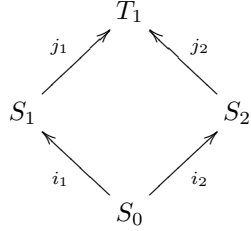
\xymatrix{
& T_1&\\ 
 S_1\ar[ur]^{j_1} & & S_2\ar[ul]_{j_2} \\
 & S_0 \ar[ul]^{i_1} \ar[ur]_{i_2}
}$$
\begin{figure}[here,top,bottom]
\caption{The lattices $S_0$, $S_1$, and $S_2$.}\label{F:lattices_S_0_S_1_S_2}
\setlength{\unitlength}{0.6mm}
\begin{picture}(90,100)(-10,-10)
\put(40,0){\line(-1,1){40}}
\put(40,0){\line(1,1){40}}

\put(0,40){\line(1,1){40}}
\put(80,40){\line(-1,1){40}}
\put(20,20){\line(1,1){40}}
\put(60,20){\line(-1,1){40}}

\put(40,0){\circle*{3}}

\put(20,20){\circle*{3}}
\put(60,20){\circle*{3}}

\put(0,40){\circle*{3}}
\put(40,40){\circle*{3}}
\put(80,40){\circle*{3}}

\put(20,60){\circle*{3}}
\put(60,60){\circle*{3}}
\put(40,80){\circle*{3}}

\put(-8,40){$a_1$}
\put(72,40){$a_4$}

\put(38,-8){$S_0$}
\end{picture}\quad
\begin{picture}(90,100)(-10,-10)
\put(40,0){\line(-1,1){40}}
\put(40,0){\line(1,1){40}}
\put(40,0){\line(0,1){80}}

\put(0,40){\line(1,1){40}}
\put(80,40){\line(-1,1){40}}
\put(20,20){\line(1,1){40}}
\put(60,20){\line(-1,1){40}}

\put(40,0){\circle*{3}}

\put(20,20){\circle*{3}}
\put(40,20){\circle*{3}}
\put(60,20){\circle*{3}}

\put(0,40){\circle*{3}}
\put(40,40){\circle*{3}}
\put(80,40){\circle*{3}}

\put(20,60){\circle*{3}}
\put(40,60){\circle*{3}}
\put(60,60){\circle*{3}}
\put(40,80){\circle*{3}}

\put(-8,40){$a_1$}
\put(72,40){$a_4$}

\put(32,20){$a_5$}
\put(32,60){$a_6$}

\put(38,-8){$S_1$}
\end{picture}
\begin{picture}(90,100)(-10,-10)
\put(40,0){\line(-1,1){40}}
\put(40,0){\line(1,1){40}}

\put(0,40){\line(1,1){40}}
\put(80,40){\line(-1,1){40}}
\put(20,20){\line(1,1){40}}
\put(60,20){\line(-1,1){40}}

\put(20,20){\line(0,1){40}}
\put(60,20){\line(0,1){40}}

\put(40,0){\circle*{3}}

\put(20,20){\circle*{3}}
\put(60,20){\circle*{3}}

\put(0,40){\circle*{3}}
\put(20,40){\circle*{3}}
\put(40,40){\circle*{3}}
\put(60,40){\circle*{3}}
\put(80,40){\circle*{3}}

\put(20,60){\circle*{3}}
\put(60,60){\circle*{3}}
\put(40,80){\circle*{3}}

\put(-8,40){$a_1$}
\put(12,40){$a_2$}

\put(52,40){$a_3$}
\put(72,40){$a_4$}

\put(38,-8){$S_2$}
\end{picture}
\end{figure}
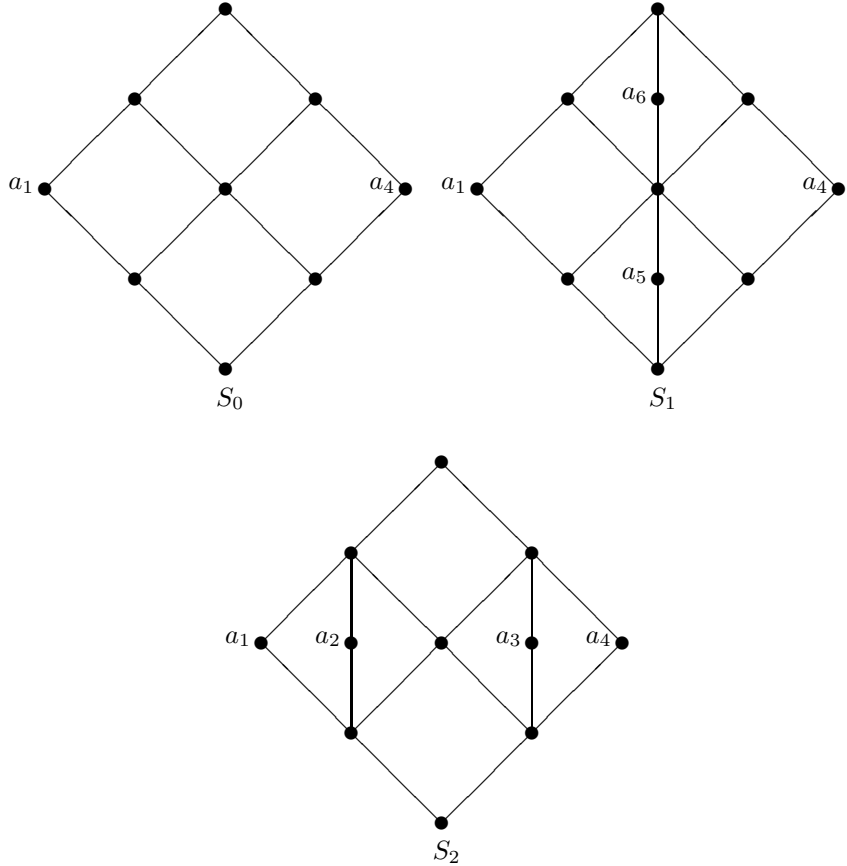

Assume that $\overrightarrow{D}$ has a lifting in~$\cV_2$:
$$ 
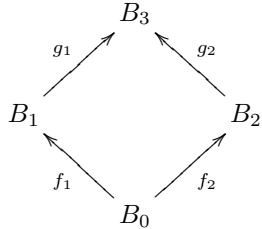
\xymatrix{
& B_3&\\ 
 B_1\ar[ur]^{g_1} & & B_2\ar[ul]_{g_2} \\
 & B_0 \ar[ul]^{f_1} \ar[ur]_{f_2}
}$$
Moreover let $(\xi_k\colon D_k\to\Con B_k)_{0\le k\le 3}$ be the corresponding isomorphism of diagrams. The \jzh s $\phi_1,\phi_2$, and $\psi$ separate $0$, thus $f_1,f_2,g_1$, and $g_2$ are one-to-one, and so we can assume that they are inclusion maps of sublattices. The lattice $B_3$ is simple, hence $B_3$ is of length at most four. As $\Conc B_0\cong 2^4$, all lattices $B_0,B_1,B_2$, and $B_3$ have length four. As $T_2,T_3,T_4,T_4-\set{a_4}$ are, up to isomorphism, all simple lattices of $\cV_2$ of length four, we can assume, by taking a larger lattice, that $B_3\in\set{T_2,T_3,T_4}$. Let $i\in\set{1,2}$. If $K$ is a sublattice of length four of $B_3$ such that $B_i\subseteq K\subseteq B_3$ and $\Conc K\cong 2^2$, by Remark~\ref{R:concsurj} the map $\Conc s\colon \Conc B_i\to \Conc K$ is surjective, where $s\colon B_i\to K$ denotes the inclusion map. Hence $\Conc s$ is an isomorphism. So, taking larger lattices, we can also assume that~$B_1$ and~$B_2$ are maximal for containment, among sublattices of~$B_3$, with respect to the property of having a congruence lattice isomorphic to~$2^2$~$(*)$.

Let $h\colon B_0\to B_1\cap B_2$, $k_1\colon  B_1\cap B_2\to B_1$, and $k_2\colon  B_1\cap B_2\to B_2$ be the respective inclusion maps. Let $\theta_1=\xi_0(1,0,0,0)$, $\theta_2=\xi_0(0,1,0,0)$, $\theta_3=\xi_0(0,0,1,0)$, and $\theta_4=\xi_0(0,0,0,1)$. So the following equalities hold:
\[(\Con f_1)(\theta_1)=(\Con f_1)(\xi_0(1,0,0,0))=\xi_1(\phi_1(1,0,0,0))=\xi_1(1,0).\]
Similarly, $(\Con f_1)(\theta_3)=(\Con f_1)(\theta_4)=\xi_1(0,1)$, so $(\Con f_1)(\theta_1)\not\le (\Con f_1)(\theta_3)$ and $(\Con f_1)(\theta_1)\not\le (\Con f_1)(\theta_4)$, but $f_1=k_1\circ h$, so $(\Con h)(\theta_1)\not\le (\Con h)(\theta_3)$ and $(\Con h)(\theta_1)\not\le (\Con h)(\theta_4)$. Moreover $(\Con f_2)(\theta_1)=\xi_1(1,0)$ and $(\Con f_2)(\theta_2)=\xi_1(0,1)$, so $(\Con h)(\theta_1)\not\le(\Con h)(\theta_2)$. Similarly, $(\Con h)(\theta_i)\not\le(\Con h)(\theta_j)$, for all $i\not=j$ in $\{1,2,3,4\}$, and so $\Con(B_1\cap B_2)$ has a four-element antichain. As $B_1\cap B_2$ modular lattice of length four, $\Con(B_1\cap B_2)\cong 2^4$.

The equalities $(\Con f_1)(\xi_0(0,0,1,1))=\xi_1(\phi_1((0,0,1,1)))=\xi_1(0,1)$ hold, so we get an embedding $B_0/\xi_0(0,0,1,1)\to B_1/\xi_1(0,1)$, but $\Con(B_0/\xi_0(0,0,1,1))\cong 2^2$, so $B_1/\xi_1(0,1)$ is a lattice of length at least two. Similarly, $B_1/\xi_1(1,0)$ is a lattice of length at least two. So all subdirectly irreducible quotients of $B_1$ have length at least two. The same holds for~$B_2$. Thus neither $B_1$ nor $B_2$ have any quotient isomorphic to~2~$(**)$.

Assume that $B_3=T_2$. As $T_2$ is generated by $a_1,a_2,a_3,a_4$, and $a_6$, which are all doubly irreducible, the maximal sublattices of $T_2$ are $T_2-\set{a_k}$, for $k\in\set{1,2,3,4,6}$, all these lattices have a congruence lattice isomorphic to $2^2$. Thus the maximal sublattices of $T_2$ with respect to the property of having a congruence lattice isomorphic to $2^2$ are $T_2-\set{a_k}$, for $k\in\set{1,2,3,4,6}$. But $T_2-\set{a_k}$ has a quotient isomorphic to $2$, for all $k\in\set{1,2,3,4}$. So by $(*)$ and $(**)$, $B_1=B_2=T_2-\set{a_6}$, thus $2^4\cong\Con (B_1\cap B_2)\cong 2^2$. So $B_3\not=T_2$. Using a dual argument we get $B_3\not=T_3$.

Assume that $B_3=T_4$. The maximal sublattices of $T_4$ with respect to the property of having a congruence lattice isomorphic to $2^2$ are $T_4-\set{a_k}$, for all $k\in\set{1,2,5,6}$, the lattice $T_4-\set{a_4,t_1}$, and the lattice $T_4-\set{a_4,t_2}$. Moreover $T_4-\set{a_5}$, $T_4-\set{a_6}$, $T_4-\set{a_4,t_1}$ and $T_4-\set{a_4,t_2}$ all have a quotient isomorphic to $2$, thus, by $(*)$ and $(**)$ both $B_1$ and $B_2$ belong to $\set{T_4-\set{a_1},T_4-\set{a_2}}$. But $\Con(T_4-\set{a_1})\cong\Con(T_4-\set{a_2})\cong\Con(T_4-\set{a_1,a_2})\cong 2^2$, which leads to a contradiction. Thus $\vec{D}$ has no lifting in~$\cV_2$. Thus it follows from Corollary~\ref{C:lifting-inFGCDvariety-with-alephn} that $\crit{\cV_1}{\cV_2}\le\aleph_1$.

All morphisms in this proof preserve~$0$ and~$1$, so
 \[
 \crit{\cV_1^{0,1}}{\cV_2^{0,1}}=\crit{\cV_1^{0}}{\cV_2^{0}} \crit{\cV_1^{0,1}}{\cV_2}=\aleph_1,
 \]
where $\cV_1^{0,1}$ (resp., $\cV_2^{0,1}$) denotes the variety of bounded lattices generated by $T_1$ (resp., $T_2,T_3$ and $T_4$); and similarly for $\cV_1^0$, and so on.
\end{proof}

\section{Conclusion}
Many of the results in this paper can be formulated in purely categorical terms, thus considerably expanding their range of application, at the expense of a noticeably heavier preparatory work. Furthermore, for a given poset $I$, the existence of a norm-covering of $I$ with properties enabling such categorical extensions gives rise to interesting combinatorial problems. These developments will be presented in a further paper.

The cardinals $\aleph_0$, $\aleph_1$, and $\aleph_2$ are critical points of some pairs of varieties of lattices, but we do not know whether there exist two finitely generated varieties of lattices with critical point~$\aleph_3$.

\section{Acknowledgment}
I thank Friedrich Wehrung for his inspiring advisorship and careful reading of the paper, and also the anonymous referee for his thoughtful report, all of which lead to many improvements both in form and substance.

\end{document}